\documentclass{amsart}[14pt,a4paper]
\usepackage{geometry}
\usepackage{amsfonts}
\usepackage{hyperref}
\usepackage{amsmath}
\usepackage{galois}
\DeclareMathSizes{10}{10}{7}{5}
\usepackage{mathtools}
\numberwithin{equation}{section}

\usepackage{enumitem}
\usepackage{amssymb}
\usepackage{mathrsfs}
\newtheorem{theorem}{Theorem}[section]
\newtheorem*{theorema}{Theorem A}
\newtheorem*{theoremb}{Theorem B}
\newtheorem*{theoremc}{Theorem C}
\newtheorem{lemma}{Lemma}[section]
\newtheorem{remark}{Remark}[section]
\newtheorem{proposition}{Proposition}[section]

\newtheorem{corollary}{Corollary}[section]

\newtheorem*{problem}{Problem}

\newcommand{\upcite}[1]{\textsuperscript{\textsuperscript{\cite{#1}}}}
\begin{document}
\title[\textbf{Area of Julia sets of non-renormalizable cubic polynomials}]
{\textbf{Area of Julia sets of non-renormalizable cubic polynomials}\protect\footnotemark[2]}
\author{Jianyong Qiao}
\author{Hongyu Qu}
\address{School of Sciences, Beijing University of Posts and Telecommunications, Beijing
100786, P. R. China. \textit{Email:} \textit{qjy@bupt.edu.cn}}
\address{School of Sciences, Beijing University of Posts and Telecommunications, Beijing
100786, P. R. China. \textit{Email:} \textit{hongyuqu@bupt.edu.cn}}
%\footnotetext[1]{National Key R\&D Program of China under Grant 2019YFB1406500}
\maketitle
\begin{abstract}
The long-standing problem of existence of nowhere dense rational Julia set with positive area has been solved by an example in quadratic polynomials by Buff and Ch\'eritat. Since then many efforts have been devoted to finding out new classes of rational maps with nowhere dense Julia sets having positive area. So far, all known examples of this kind are renormalizable with only one exception which is a quadratic polynomial. In this paper, by developing a new approach, we prove that there exists a non-renormalizable cubic polynomial having a Julia set with positive area.
\end{abstract}

\section{Introduction and main results}
This paper is mainly concerned with Lebesgue measure problem of Julia sets of polynomials, which can be traced back to the study of Fatou\upcite{F1919}. Fatou gave the first example of Julia set with zero area. In the 1980s and 1990s, a large number of examples of Julia sets with zero area were found (see \cite{DH84}\cite{DH85}\cite{McM}\cite{Lyu83}\cite{Lyu84}\cite{Lyu}\cite{PZ}\cite{Shi1}). On the other hand, Douady, Buff, Nowicki and Strien also tried to construct polynomial Julia sets with positive area (see \cite{B97} and \cite{C1}). It was not until the beginning of this century that Buff and Ch\'eritat constructed the first example of nowhere dense rational Julia set with positive area based on the plan of Douady\upcite{C1}. In fact, they constructed three types of quadratic polynomials having Julia sets with positive area, which are Cremer quadratic polynomials, Siegel quadratic polynomials and infinite renormalizable quadratic polynomials with unbounded satellite combinatorics. In recent years, Avila and Lyubich\upcite{AL} constructed infinite renormalizable quadratic polynomials with bounded primitive combinatorics having Julia sets with positive area\footnote{This is the first example of locally connected quadratic Julia set with positive area.}; Dudko and Lyubich\upcite{DL18} constructed infinite renormalizable quadratic polynomials with bounded satellite combinatorics having Julia sets with positive area.

We observe that based on the results on quadratic polynomials and the Douady-Hubbard theory of polynomial-like mappings, for all $d\geq3$, one can construct a polynomial of degree $d$ having a Julia set with positive area.
\footnote{This idea had been used by Y. Fu and F. Yang\upcite{FY} to construct rational Sierpi\'nski carpet Julia sets with positive area.}
Let $P_{\theta}:\mathbb{C}\to\mathbb{C}$ be a quadratic polynomial
\[P_{\theta}(z)=e^{2\pi i\theta}z+z^2,\]
where $\theta$ has continued fraction expansion
\[\theta=[0,a_1,a_2,\cdots]=\frac{1}{a_1+\frac{1}{a_2+\frac{1}{a_3+\cdots}}}.\]
According to \cite{BC}, we have

\begin{theorema}[Buff and Ch\'eritat]
\label{T139}There exists a dense subset $\mathcal{S}$ of $\mathbb{R}/\mathbb{Z}$ such that for all $\theta\in\mathcal{S}$,
$f_{\theta}$ has a Julia set with positive area.
\end{theorema}

In fact, according to \cite{BC}, there exist two nonempty dense subsets $\mathcal{S}_1,\mathcal{S}_2$ of $\mathcal{S}$
with $\mathcal{S}=\mathcal{S}_1\cup\mathcal{S}_2$ such that for all $\theta\in\mathcal{S}_1$
(or $\theta\in\mathcal{S}_2$), the origin is a Cremer (or Siegel) point of $P_{\theta}$
and Julia set $J(P_{\theta})$ of $P_{\theta}$ has positive area.

For all $\theta\in\mathbb{R}/\mathbb{Z}$, $\epsilon>0$ and positive integer $d\geq3$, we set
\[f_{\theta,\epsilon,d}(z)=P_{\theta}(z)+\epsilon z^d\ (z\in\mathbb{C}).\]
Then we have

\begin{proposition}
\label{p139}For every positive integer $d\geq3$ and enough small $\epsilon$, there exist $\theta\in\mathcal{S}$ and Jordan regions $U,V\subset\mathbb{C}$ such that $U$ is compactly contained in $V$, $f_{\theta,\epsilon,d}|_U:U\to V$ is a quadratic-like mapping and $f_{\theta,\epsilon,d}|_U$ is hybrid equivalent to $P_{\theta}$.
\end{proposition}

Combining Theorem A and Proposition \ref{p139}, for every positive integer $d\geq3$,
there exit $\theta\in\mathbb{R}/\mathbb{Z}$ and $\epsilon>0$ such that $f_{\theta,\epsilon,d}$ has a small
Julia set with positive area. Thus $f_{\theta,\epsilon,d}$ has a Julia set with positive area.
So far, all known rational maps with nowhere dense Julia sets having positive area are quadratic or renormalizable. Recall that a rational map $f$ of degree greater than $1$ is said to be renormalizable if there exist a positive integer $p$ and two Jordan regions $U$ and $V$ such that $U$ is compactly contained in $V$ and the restriction $f^p|_U:U\to V$ is a quadratic-like mapping with connected Julia set.
Moreover, if $f$ is a quadratic polynomial, it is requested that $p\geq2$.
Noting that dynamical systems of renormalizable rational maps are similar to dynamical systems of quadratic polynomials at a small scale,
we thus have the following natural problem.

\begin{problem}
Does there exist a non-renormalizable polynomial of degree greater than $2$ such that its Julia set has positive area?
\end{problem}

In this paper, we give a positive answer to this problem. Let $f_{\theta}$ be a cubic polynomial
\[f_{\theta}(z)=e^{2\pi i\theta}z(z+1)^2,\]
where $\theta\in\mathbb{R}/\mathbb{Z}$. We first have

\begin{proposition}
\label{p239}For all $\theta\in\mathbb{R}/\mathbb{Z}$, the cubic polynomial $f_{\theta}$ is non-renormalizable.
\end{proposition}

Our main result is the following theorem.

\begin{theorem}
\label{T239}There exists $\theta\in\mathbb{R}/\mathbb{Z}$ such that the non-renormalizable cubic polynomial $f_{\theta}$ has the Cremer point $0$ and the Julia set with positive area.
\end{theorem}

\begin{remark}
{\rm Since quadratic polynomials with an indifferent fixed point are non-renormalizable, the results of Buff and Ch\'eritat tell us the existence of non-renormalizable quadratic polynomials having Julia sets with positive area.}
\end{remark}

Furthermore we define a family of cubic polynomials
\[\mathcal{F}=\{e^{2\pi i\theta}z(z+1)^2|\theta\in\mathbb{R}/\mathbb{Z}\}.\]
Our main point of view on the proof of Theorem \ref{T239} is to realize the method of constructing quadratic Cremer Julia sets with positive area by Buff and Ch\'eritat on the family $\mathcal{F}$.
The main difficulty of applying this method on $\mathcal{F}$ is how to control the limit area of the perturbed Siegel disks of the family $\mathcal{F}$. Let us recall the control of Buff and Ch\'eritat on the limit area of the perturbed Siegel disks of the family of quadratic polynomials.
Assume that $\theta$ is a Brjuno number and $\{\theta_n\}_{n=1}^{+\infty}$ is a sequence of Brjuno numbers such that $\theta_n\to\theta$ as $n\to+\infty$. According to \cite{Yo}, $f_{\theta}$ and $f_{\theta_n}$($n\geq1$) have Siegel disks at 0. We denote the Siegel disk of $f_{\theta}$ at 0 by $\Delta_{f_{\theta}}$ and
the Siegel disk of $f_{\theta_n}$ at 0 by $\Delta_{f_{\theta_n}}$($n\geq1$). It is well known that the dynamical system of $f_{\theta}$ depends extremely sensitively on $\theta$. Generally, there exists a drastic difference between asymptotic properties of the sequence
$\{\Delta_{f_{\theta_n}}\}_{n=1}^{+\infty}$ of Siegel disks and properties of the Siegel disk $\Delta_{f_{\theta}}$.
The result of Buff and Ch\'eritat is to give a nice choice of $\{\theta_n\}_{n=1}^{+\infty}$ so that they can control the asymptotic area of the sequence $\{\Delta_{f_{\theta_n}}\}_{n=1}^{+\infty}$.
If $U$ and $X$ are measurable subsets of $\mathbb{C}$, with $0<{\rm area}(U)<+\infty$, we use the notation
\[{\rm dens}_{U}(X):=\frac{{\rm area}(U\cap X)}{{\rm area}(U)}.\]
Then we formula the result as follows.

\begin{theoremb}[Buff and Ch\'eritat]
\label{T1.1}Assume $\alpha:=[a_0,a_1,a_2,\cdots]$ and $\theta:=[0,t_1,t_2,\cdots]$ are Brjuno numbers and let $p_n/q_n$ be the approximants to $\alpha$. Assume
\[\alpha_n:=[a_0,a_1,\cdots,a_n,A_n,t_1,t_2,\cdots]\]
with $(A_n)$ a sequence of positive integers such that
\begin{equation}
\label{F1.1}\limsup\limits_{n\to+\infty}\sqrt[q_n]{\log A_n}\leq1.
\end{equation}
Let $\Delta$ be the Siegel disk of $P_{\alpha}$ and $\Delta_n'$ be the Siegel disk of the restriction of $P_{\alpha_n}$ to $\Delta$. For any nonempty open set $U\subset\Delta$,
\[\liminf\limits_{n\to+\infty}{\rm dens}_{U}(\Delta_n')\geq\frac{1}{2}.\]
\end{theoremb}
\begin{remark}
{\rm In \cite{BC}, the condition (\ref{F1.1}) is a technical condition and is not necessary.}
\end{remark}

Since the method of Buff and Ch\'eritat needs to use the parabolic explosion technique of Ch\'eritat, and this technique is established for quadratic polynomials, we can not directly use their method to establish a similar version of Theorem B on the family $\mathcal{F}$.
In this paper,
we first transfer the result of Buff and Ch\'eritat about quadratic polynomials to the class of Inou-Shishikura based on the near-parabolic renormalization operator's uniform contraction
in fiber directions on the class of Inou-Shishikura with respect to
Teichm\"uller distance and uniform area distortion of quasiconformal mappings,
then based on the connection between the class of Inou-Shishikura and the family $\mathcal{F}$,
we establish a similar version of Theorem B on the family $\mathcal{F}$.
Before giving this result, we first recall the class of Inou-Shishikura.
First, consider the ellipse
\[E:=\left\{x+iy\in\mathbb{C}:\left(\frac{x+0.18}{1.24}\right)^2+\left(\frac{y}{1.04}\right)^2\leq1\right\}\]
and the map $g(z):=-\frac{4z}{(1+z)^2}$. Then
the polynomial $P(z):=z(1+z)^2$ restricted to the domain
$V:=g(\widehat{\mathbb{C}}\setminus E)$,
has a fixed point at $0\in V$ with multiplier $1$, a critical point at $-1/3\in V$ that is mapped to $-4/27$. It has another critical point at $-1\in\mathbb{C}\setminus V$ with $P(-1)=0$.
Following \cite{IS}, we define the class of maps
\[\mathcal{IS}_0:=\left\{P\comp\phi^{-1}:\phi(V)\to\mathbb{C}
\left|\begin{matrix}\phi\ {\rm is\ univalent},\
\phi(0)=0,\ \phi'(0)=1,\ {\rm and}\\ \phi\ {\rm has\ a\
quasiconformal\ extension\ to}\ \mathbb{C}\end{matrix}\right.\right\}.\]
For all $f=P\comp\phi^{-1}\in\mathcal{IS}_0$, 0 is a parabolic fixed point of $f$ and $\phi(-1/3)$ is the only critical point of $f$ with the critical value $-4/27$.
Fix a positive integer $N$, let $\mathbb{HT}_N$ denote the set of irrational number $\alpha=[a_0,a_1,a_2,\cdots]$ with $a_j\geq N$ ($j=1,2,\cdots$).
Our result is the following theorem.

\begin{theorem}
\label{T20}Assume $\alpha:=[a_0,a_1,a_2,\cdots]$ and $\theta:=[0,t_1,t_2,\cdots]$ are Brjuno numbers
and let $p_n/q_n$ be the approximants to $\alpha$. Set
\[\alpha_n:=[a_0,a_1,\cdots,a_n,A_n,t_1,t_2,\cdots]\]
with $(A_n)$ a sequence of positive integers such that
\begin{equation*}
\limsup\limits_{n\to+\infty}\sqrt[q_n]{\log A_n}\leq1.
\end{equation*}
Then there exists a positive integer $N$ such that
if $\alpha\in\mathbb{HT}_N$ and $\theta\in\mathbb{HT}_N$, then\\
for all $h\in\mathcal{IS}_0$ and any nonempty open set $U\subset\Delta_{f_{\alpha}}$, we have
\[\liminf\limits_{n\to+\infty}{\rm dens}_U(\Delta'_{f_{\alpha_n}})\geq\frac{1}{2},\]
where $\Delta_{f_{\alpha}}$ is the Siegel disk of
$f_{\alpha}=e^{2\pi i\alpha}h$ centering at $0$ and $\Delta'_{f_{\alpha_n}}$ is the Siegel disk of
the restriction of $f_{\alpha_n}=e^{2\pi i\alpha_n}h$ to
$\Delta_{f_{\alpha}}$ centering at $0$.
\end{theorem}

\section{Complex analytic structure and renormalization of the class of Inou-Shishikura\label{S5.2}}
Let
\[\mathcal{S}^{qc}(V):=\left\{\phi:V\to\mathbb{C}\left|\begin{matrix}\phi\ {\rm is\
univalent\ with}\ \phi(0)=0,\ \phi'(0)=1\\ {\rm and\ has\ a\ quasiconformal\ extension\ to}\ \mathbb{C}
\end{matrix}\right.\right\}.\]
Let $W=\mathbb{C}\setminus\overline{V}$, which is a punctured disk in $\hat{\mathbb{C}}$.
We denote by ${\rm Teich}(W)$ the standard Teichm\"uller space of $W$.
It is well known that ${\rm Teich}(W)$ is a complete metric space with respect to the Teichm\"uller metric
and has a complex Banach manifold structure due to Bers.
We define a map
$$R:\mathcal{S}^{qc}(V)\to {\rm Teich}(W),\ \phi\mapsto[\hat{\phi}|_W],$$
where $\hat{\phi}:\mathbb{C}\to\mathbb{C}$ is a quasiconformal extension of $\phi$.
According to \cite{IS}, $R$ is well defined and is a bijection.
Then for any $\phi_1,\phi_2\in\mathcal{S}^{qc}(V)$, we define
the distance $d_T(\phi_1,\phi_2)$ of $\phi_1$ and $\phi_2$ as the Teichm\"uller distance of $R(\phi_1)$ and $R(\phi_2)$. Following \cite{IS}, we have the following lemma.
\begin{lemma}[Inou and Shishikura]
\label{l2}If $\phi_n,\phi\in\mathcal{S}^{qc}(V)$ and $d_T(\phi_n,\phi)\to0${\rm(}$n\to+\infty${\rm)},
then $\phi_n$ converges to $\phi$ uniformly on every compact sets in $V$ as $n\to+\infty$.
A mapping $\tau$ from a complex manifold $\Lambda$ to ${\rm Teich}(W)$ is holomorphic if and only if
there exists a holomorphic function $\phi:\Lambda\times V\to\mathbb{C}$ such that
$\phi_{\lambda}:=\phi(\lambda,\cdot)\in\mathcal{S}^{qc}(V)$ and $R(\phi_{\lambda})=\tau(\lambda)$.
\end{lemma}

Next, we introduce the near-parabolic renormalization on perturbations of the class of Inou-Shishikura.
According to Inou and Shishikura\upcite{IS},
there exists an $\alpha^*>0$ such that for any $0<\alpha<\alpha^*$ and $h=P\comp\phi^{-1}\in\mathcal{IS}_0$
with $\phi(0)=0$ and $\phi'(0)=1$, the following four properties hold.

\vspace{0.1cm}
\noindent(1) $h(e^{2\pi i\alpha}z)$ has a non-zero
fixed point $\sigma_{\alpha}$ near $0$ in $e^{-2\pi i\alpha}\phi(V)$. The fixed point $\sigma_{\alpha}$ depends continuously on $\alpha$, and has asymptotic expansion $\sigma_{\alpha}=-4\pi\alpha i/h''(0)+o(\alpha)$,
when $\alpha$ converges to 0 in a fixed neighborhood of $0$.

\vspace{0.1cm}
Let
\[Q_{\alpha}(z):=e^{2\pi i\alpha}z+\frac{27}{16}e^{4\pi i\alpha}z^2.\]
We observe that the quadratic polynomial $Q_{\alpha}$ has only finite critical value at $-4/27$.
Let $g(z)=h(e^{2\pi i\alpha}z)$ or $g=Q_{\alpha}$. Then $g$ has only finite critical value at $-4/27$
in its domain ${\rm Def}(g)$, where ${\rm Def}(g)=e^{-2\pi i\alpha}\phi(V)$ if $g(z)=h(e^{2\pi i\alpha}z)$
and ${\rm Def}(g)=\mathbb{C}$ if $g=Q_{\alpha}$. Denote by $\sigma_g$ the nonzero fixed point of $g$ and
by $c_g$ the critical point of $g$ in ${\rm Def}(g)$.

\vspace{0.1cm}
\noindent(2) There exist a domain $\mathcal{P}_g\subset{\rm Def}(g)$ and a univalent map $\Phi_g:\mathcal{P}_g\to\mathbb{C}$ satisfying the following properties:

\begin{itemize}
\item The domain $\mathcal{P}_g$ is bounded by piecewise smooth curves and is compactly contained in
${\rm Def}(g)$. Moreover, it contains $c_g$, $0$ and $\sigma_g$ on its boundary.
\item ${\rm Im}\ \Phi_g(z)\to+\infty$ when $z\in\mathcal{P}_g\to0$, and ${\rm Im}\ \Phi_g(z)\to-\infty$ when $z\in\mathcal{P}_g\to\sigma_g$.
\item $\Phi_g(g(z))=\Phi_g(z)+1$, whenever $z$ and $g(z)$ belong to $\mathcal{P}_g$.
\item $\Phi_g$ is uniquely determined by the above conditions together with normalization $\Phi_g(c_g)=0$. Moreover, the normalized $\Phi_g$ depends continuously on $g$.
\end{itemize}
\noindent Furthermore, it is proved in \cite{BC} that there is a positive integer ${\bf k}$, independent of $g$,
such that
%for all $g\in\mathcal{IS}_\alpha\cup\{Q_{\alpha}\}$, $\Phi_g$ satisfies
\begin{itemize}
\item $\Phi_g(\mathcal{P}_g)=\{w\in\mathbb{C}|0<{\rm Re}(w)<1/\alpha-{\bf k}\}$.
\end{itemize}

The map $\Phi_g:\mathcal{P}_g\to\mathbb{C}$ is called the perturbed Fatou coordinate of $g$ and $\mathcal{P}_g$ is called the perturbed petal of $g$. Define
\begin{align*}
\mathcal{C}_g&:=\{z\in\mathcal{P}_g:1/2\leq {\rm Re}(\Phi_g(z))\leq 3/2, -2\leq{\rm Im}(\Phi_g(z))\leq2\},\\
\mathcal{C}_g^{\#}&:=\{z\in\mathcal{P}_g:1/2\leq {\rm Re}(\Phi_g(z))\leq 3/2, {\rm Im}(\Phi_g(z))\geq2\}.
\end{align*}
By definition, $-4/27\in{\rm int}(\mathcal{C}_g)$ and $0\in\partial(\mathcal{C}_g^{\#})$.
We call the arc
$$\{z\in\mathcal{P}_g:{\rm Re}(\Phi_g(z))=1/2, {\rm Im}(\Phi_g(z))\geq-2\}$$
the left boundary
of $\mathcal{C}_g\cup\mathcal{C}_g^{\#}$, the arc
$$\{z\in\mathcal{P}_g:{\rm Re}(\Phi_g(z))=3/2, {\rm Im}(\Phi_g(z))\geq-2\}$$
the right boundary
of $\mathcal{C}_g\cup\mathcal{C}_g^{\#}$, the arc
$$\{z\in\mathcal{P}_g:1/2\leq{\rm Re}(\Phi_g(z))\leq3/2, {\rm Im}(\Phi_g(z))=-2\}$$
the down end of $\mathcal{C}_g\cup\mathcal{C}_g^{\#}$ and the point 0 the `0' end of
$\mathcal{C}_g\cup\mathcal{C}_g^{\#}$ respectively.

\vspace{0.1cm}
\noindent(3) It is proved in \cite{BC} that there exists a smallest positive integer ${\bf k}_1$, independent of $g$, with the following properties:
\begin{itemize}
\item For every integer $k$, with $0\leq k\leq{\bf k}_1$, there exists a unique connected component of $g^{-k}(\mathcal{C}_g^{\#})$ which is compactly contained in ${\rm Def}(g)$, and contains $0$ on its boundary.
    We denote by $(\mathcal{C}_g^{\#})^{-k}$ this component.
\item For every integer $k$, with $0\leq k\leq{\bf k}_1$, there exists a unique connected component of $g^{-k}(\mathcal{C}_g)$ which has non-empty intersection with $(\mathcal{C}_g^{\#})^{-k}$, and is compactly contained in ${\rm Def}(g)$. This component is denoted by $\mathcal{C}_g^{-k}$.
\item The sets $\mathcal{C}_g^{-{\bf k}_1}$ and $(\mathcal{C}_g^{\#})^{-{\bf k}_1}$ are contained in
\[\{z\in\mathcal{P}_g|2<{\rm Re}\ \Phi_g(z)<1/\alpha-{\bf k}-2\}.\]
\item The maps $g:\mathcal{C}_g^{-k}\to\mathcal{C}_g^{-k+1}$ ($2\leq k\leq{\bf k}_1$) and $g:(\mathcal{C}_g^{\#})^{-k}\to(\mathcal{C}_g^{\#})^{-k+1}$ ($1\leq k\leq{\bf k}_1$) are univalent.
    The map $g:\mathcal{C}_g^{-1}\to\mathcal{C}_g$ is a degree two branched covering.
\end{itemize}

Define
\[\mathcal{D}_g^{-{\bf k}_1}:=\mathcal{C}_g^{-{\bf k}_1}\cup(\mathcal{C}_g^{\#})^{-{\bf k}_1}\
{\rm and}\
{\rm Exp}(\zeta):=s(\frac{-4}{27}e^{2\pi i\zeta}),\]
where $s(z)=\overline{z}$. The renormalization of $g$ is defined by
the map
\[\mathcal{R}(g)(z)={\rm Exp}\comp\Phi_g\comp g^{{\bf k}_1}\comp\Phi_g^{-1}\comp{\rm Exp}^{-1}(z),\
z\in{\rm Exp}\comp\Phi_g(\mathcal{D}_g^{-{\bf k}_1}).\]
It is easy to check that $\mathcal{R}(g)$ is well-defined, can be extended at $0$, has the form
$e^{2\pi\frac{1}{\alpha}i}z+O(z^2)$ near $0$ and the critical value at $-4/27$.

\vspace{0.1cm}
\noindent(4) There exists a Jordan domain $U\supset\overline{V}$(independent of $g$) and a univalent map $\phi:U\to\mathbb{C}$ with $\phi(0)=0$ and $\phi'(0)=1$
such that
%the domain of $\mathcal{R}(g)$ contains $e^{-2\pi i\frac{1}{\alpha}}\phi{V}$ and
\[\mathcal{R}(g)(z)=P\comp\phi^{-1}(e^{2\pi i\frac{1}{\alpha}}z),\
z\in e^{-2\pi i\frac{1}{\alpha}}\phi(U).\]

The restriction $\mathcal{R}(g)$ to $e^{-2\pi i\frac{1}{\alpha}}\phi(V)$ is called the near-parabolic renormalization
of $g$ by Inou and Shishikura. And it is clear that
$\mathcal{R}(g)(e^{-2\pi i\frac{1}{\alpha}}z)$ ($z\in\phi(V)$) belongs to $\mathcal{IS}_0$.
Following \cite{IS}, we have the following theorem.

\begin{theoremc}[Inou and Shishikura]
\label{T2.1}There exists $0<\lambda<1$ such that for all $\alpha\in(0,\alpha^*]$ and
$h_1,h_2\in\mathcal{IS}_0$ with $h_1=P\comp\psi_1^{-1}$ and $h_2=P\comp\psi_2^{-1}$, we have
\[d_T(\phi_1,\phi_2)\leq\lambda d_T(\psi_1,\psi_2),\]
where
$\mathcal{R}(h_1\comp e^{2\pi i\alpha})=P\comp\phi_1^{-1}\comp e^{2\pi i\frac{1}{\alpha}}$
and
$\mathcal{R}(h_2\comp e^{2\pi i\alpha})=P\comp\phi_2^{-1}\comp e^{2\pi i\frac{1}{\alpha}}.$
\end{theoremc}

\vspace{0.1cm}
\noindent For convenience, we give some notations as the following.
\begin{itemize}
\item Let $N$ be a positive integer such that $0<1/N<\alpha^*$ and $N>{\bf k}+{\bf k}_1+1$.
\item $\mathcal{IS}_N=\{h\comp e^{2\pi i\alpha}|h\in\mathcal{IS}_0\ {\rm and}\ \alpha\in\mathbb{HT}_N\}$.
\item For all $f\in\mathcal{IS}_N$, let $0<\alpha(f)<1$ denote the rotation number of $f$ at the origin.
Then there exists $\psi\in\mathcal{S}^{qc}$ such that $f=P\comp\psi^{-1}\comp e^{2\pi i\alpha(f)}$. We let $V_f$ denote $e^{-2\pi i\alpha(f)}\psi(V)$.
\item For all $\alpha\in\mathbb{HT}_N$, let $\mathcal{IS}_{\alpha}$ be the subset of $\mathcal{IS}_N$
      consisting of $f\in\mathcal{IS}_N$ with $\alpha(f)=\alpha$.
\item For all $f\in\mathcal{IS}_N\cup\{Q_{\alpha}\}$, the notation $c_f$ and $cv_f$ denote
      the only finite critical point and critical value $-4/27$ respectively,
      the notation $\sigma(f)$ denotes the only finite nonzero fixed point of $f$, the domain
      $\mathcal{P}_f$ denotes the perturbed Fatou petal and
      the map $\Phi_f$ denotes the perturbed Fatou coordinate.
\item For all $f,g\in\mathcal{IS}_N$ with $f=P\comp\psi_f^{-1}\comp e^{2\pi i\alpha(f)}$ and
      $g=P\comp\psi_g^{-1}\comp e^{2\pi i\alpha(g)}$, we let $d_T(f,g)=d_T(\psi_f,\psi_g)$.
\item For all $f\in\mathcal{IS}_N$, let $\mathcal{O}_f$ denote the postcritical set of $f$ in $V_f$, i.e.,
      \[\mathcal{O}_f=\overline{\{w\in V_f|\exists\ n\geq1\ s.t.\ f^n(c_f)=w\ \&\
      f^k(c_f)\in V_f\ (1\leq k\leq n)\}}.\]
      If $f$ is a quadratic polynomial, $\mathcal{O}_f$ denote the postcritical set of $f$ in $\mathbb{C}$.
\item For all Brjuno number $\alpha\in\mathbb{HT}_N$, if $f\in\mathcal{IS}_{\alpha}$,
      the notation $\Delta_f$ denotes the Siegel disk of $f$ centering at 0 in $V_f$ and if $f=Q_{\alpha}$,
      the notation $\Delta_f$ denotes the Siegel disk of $f$ centering at 0.
\end{itemize}

\vspace{0.1cm}
\noindent At the end of this section we give the following two lemmas.
\begin{lemma}[\cite{BC}]
\label{BCL1}Assume that $\alpha\in\mathbb{HT}_N$ is a Brjuno number and $h\in\mathcal{IS}_{\alpha}\cup\{Q_{\alpha}\}$.
Let $w\in V_{\mathcal{R}(h)}$ and $w'=\mathcal{R}(h)(w)$.
If $z\in\mathcal{P}_h$ and $z'\in\mathcal{P}_h$ satisfy
\[{\rm Exp}\comp\Phi_h(z)=w\ and\ {\rm Exp}\comp\Phi_h(z')=w',\]
then there is an integer $k\geq1$ such that $z'=h^k(z)$ and $h^l(z)\in{\rm Def}(h)$ for $l=0,1,\cdots,k$.
\end{lemma}
\begin{lemma}[\cite{IS}]
\label{L3.1}For any map $f\in\mathcal{IS}_N$, the critical point is
non-escaping{\rm(}i.e., $f^n(c_f)\in V_f$, $n=0,1,...${\rm)}
and stays away from the boundary of $V_f$. Thus, the postcritical set $\mathcal{O}_f$ is compactly
contained in $V_f${\rm(}uniformly over $\mathcal{IS}_N${\rm)}.
In general, the critical point $c_f$ is not eventually periodic.
\end{lemma}

\section{Holomorphic motions}
Let $X$ be a complex analytic manifold with a base point $\lambda_0\in X$, and let $E\subset\hat{\mathbb{C}}$ be a subset.
A holomorphic motion of $E$, parametrized by $X$, is a map
$\phi:X\times E\to \hat{\mathbb{C}}$ such that
\begin{itemize}
\item for each $x\in E$, the map $\lambda\mapsto\phi(\lambda,x)$ is analytic,
\item for each $\lambda\in X$, the map $x\mapsto\phi(\lambda,x)$ is injective,
\item $\phi(\lambda_0,x)=x$.
\end{itemize}
Following \cite{BR}, we have the following lemma on holomorphic motions.

\begin{lemma}[Bers and Royden]
\label{l1}Let $E$ be a subset of $\hat{\mathbb{C}}$ and
$\phi:\mathbb{D}\times E\to\hat{\mathbb{C}}$ is a holomorphic motion of $E$ with a base point 0.
Then every $\phi(z,\cdot)$ is the restriction to $E$ of a quasiconformal self-map
$\hat{\phi}(z,\cdot)$ of $\hat{\mathbb{C}}$, of dilatation not exceeding
$K=\frac{1+|z|}{1-|z|}$.
\end{lemma}
We will prove the following two lemmas by means of holomorphic motions.

\begin{lemma}
\label{L3.2}Assume $\alpha\in\mathbb{HT}_N$ and $f\in\mathcal{IS}_{\alpha}$. If $f$ has a Siegel disk $\Delta_f$
at the origin, then $\Delta_f$ is the component of $V_f\setminus\mathcal{O}_f$ containing $0$ and compactly contained in $V_f$.
\end{lemma}

\begin{proof}
Since $f\in\mathcal{IS}_{\alpha}$, there exists $\psi\in\mathcal{S}^{qc}$ such that
$f=P\comp\psi^{-1}\comp e^{2\pi i\alpha}$.
%Let $\pi(f)=f\comp e^{-2\pi i\alpha}$.
Since both $R(id|_V)$ and $R(\psi)$ belong to ${\rm Teich}(W)$, there exists a holomorphic map
$\tau:\mathbb{D}\to{\rm Teich}(W)$ such that $\tau(0)=R(id|_V)$ and $\tau(a)=R(\psi)$ for some
$a\in\mathbb{D}$. By  Lemma \ref{l2}, there exists a holomorphic function
$\phi:\mathbb{D}\times V\to\mathbb{C}$ such that $\phi_{\lambda}:=\phi(\lambda,\cdot)\in\mathcal{S}^{qc}(V)$
and $R(\phi_{\lambda})=\tau(\lambda)$.
Let $\tilde{\tau}(\lambda)=P\comp\phi_{\lambda}^{-1}\comp e^{2\pi i\alpha}$ for $\lambda\in\mathbb{D}$ and $c_{\lambda}$ is the critical point of $\tilde{\tau}(\lambda)$, that is,
$c_{\lambda}=e^{-2\pi i\alpha}\phi_{\lambda}(-\frac{1}{3})$.
Then $\tilde{\tau}(0)=P\comp e^{2\pi i\alpha}|_{e^{-2\pi i\alpha}V}$ and $\tilde{\tau}(a)=f$.
Let
$$E=\{0,\infty\}\cup\{\tilde{\tau}(0)^k(c_0)|k=0,1,\cdots\}$$
and
\[H:\mathbb{D}\times E\to\mathbb{C}\]
such that
\[H(\lambda,x)=\left\{\begin{matrix}\tilde{\tau}(\lambda)^k(c_{\lambda}),
&x=\tilde{\tau}(0)^k(c_0),k\geq0\\
x,&x=0,\infty\end{matrix}\right..\]
It is easy to check that for all $x\in E$, $H(\cdot,x)$ is a holomorphic function.
By Lemma \ref{L3.1}, for all $\lambda\in\mathbb{D}$, $H(\lambda,\cdot)$ is injective.
Thus the map $H$ is a holomorphic motion with the base point $0$. By Lemma \ref{l1},
there exists a quasiconformal map $\phi$ (with dilatation not exceeding $\frac{1+|a|}{1-|a|}$)
from $\hat{\mathbb{C}}$ to itself such that $\phi(z)=H(a,z)$ for $z\in E$.
Thus $\phi(\mathcal{O}_{\tilde{\tau}(0)})=\mathcal{O}_{\tilde{\tau}(a)}$ and
for all $z\in\mathcal{O}_{\tilde{\tau}(0)}\cup\{0,\infty\}$,
$\phi\comp\tilde{\tau}(0)(z)=\tilde{\tau}(a)\comp\phi(z)$, that is,
$\phi(\mathcal{O}_{P\comp e^{2\pi i\alpha}|_{e^{-2\pi i\alpha}V}})=\mathcal{O}_f$
and
for all $z\in\mathcal{O}_{P\comp e^{2\pi i\alpha}|_{e^{-2\pi i\alpha}V}}\cup\{0,\infty\}$,
$\phi\comp P\comp e^{2\pi i\alpha}|_{e^{-2\pi i\alpha}V}(z)=f\comp\phi(z)$.
Since $0\not\in\mathcal{O}_f=\phi(\mathcal{O}_{P\comp e^{2\pi i\alpha}|_{e^{-2\pi i\alpha}V}})$,
we have $0\not\in\mathcal{O}_{P\comp e^{2\pi i\alpha}|_{e^{-2\pi i\alpha}V}}$.
Thus the origin is a Siegel point of $P\comp e^{2\pi i\alpha}|_{e^{-2\pi i\alpha}V}$.
Since
$$\partial(\phi(\Delta_{P\comp e^{2\pi i\alpha}|_{e^{-2\pi i\alpha}V}}))=\phi(\partial\Delta_{P\comp e^{2\pi i\alpha}|_{e^{-2\pi i\alpha}V}})\subset
\phi(\mathcal{O}_{P\comp e^{2\pi i\alpha}|_{e^{-2\pi i\alpha}V}})\subset\mathcal{O}_f\subset V_f,$$
we have that $\phi(\Delta_{P\comp e^{2\pi i\alpha}|_{e^{-2\pi i\alpha}V}})$ is compactly contained in $V_f$ and
is the simply connected component of $V_f\setminus\mathcal{O}_f$ containing $0$.
Since
$$f(\phi(\partial\Delta_{P\comp e^{2\pi i\alpha}|_{e^{-2\pi i\alpha}V}}))=
\phi(P\comp e^{2\pi i\alpha}|_{e^{-2\pi i\alpha}V}(\partial\Delta_{P\comp e^{2\pi i\alpha}|_{e^{-2\pi i\alpha}V}}))
=\phi(\partial\Delta_{P\comp e^{2\pi i\alpha}|_{e^{-2\pi i\alpha}V}}),$$
we have $\phi(\Delta_{P\comp e^{2\pi i\alpha}|_{e^{-2\pi i\alpha}V}})\subset\Delta_f$.
It is known that $c_0$ is a recurrent critical point. Then $c_f$ is a recurrent critical point.
Thus $\Delta_f\subset\phi(\Delta_{P\comp e^{2\pi i\alpha}|_{e^{-2\pi i\alpha}V}})$.
This completes the proof.
\end{proof}

\begin{lemma}
\label{L4.3}For all $\alpha\in\mathbb{HT}_N$ and $f,g\in\mathcal{IS}_{\alpha}$, there exists a quasiconformal map $\phi$ from $\hat{\mathbb{C}}$ onto itself with dilatation not exceeding $e^{d_{T}(f,g)}$ such that
$\phi(0)=0$, $\phi(\mathcal{O}_f)=\mathcal{O}_g$ and $\phi\comp f(z)=g\comp\phi(z)$ for $z\in\mathcal{O}_f$. Consequently, if $\alpha$ is a Brjuno number and $\Delta_f$ (or $\Delta_g$) is the Siegel disk of $f$ (or $g$)
at the origin, then $\phi(\Delta_f)=\Delta_g$. Moreover, we can further request that $\phi$ is conformal on $\Delta_f$ and
%conjugates $f$ on $\Delta_f$ to $g$ on $\Delta_g$
$\phi\comp f(z)=g\comp\phi(z)$ for $z\in\Delta_f$.
\end{lemma}

\begin{proof}
Since $f,g\in\mathcal{IS}_{\alpha}$, there exist $\psi_0, \psi_1\in\mathcal{S}^{qc}$ such that
$f=P\comp\psi_0^{-1}\comp e^{2\pi i\alpha}$ and $g=P\comp\psi_1^{-1}\comp e^{2\pi i\alpha}$.
According to Royden-Gardiner Theorem\upcite{GA}, complex manifold ${\rm Teich}(W)$ is Kobayashi hyperbolic and the Kobayashi distance coincides with the Teichm\"uller distance.
Thus we can choose a sequence of subsets
$A_n=\{h_{nj}\}_{j=1}^{t_n}\subset\mathcal{S}^{qc}$ such that for every $n$, there exist $t_n+1$ complex numbers
$a_{n1},\cdots,a_{nt_n},a_{n(t_n+1)}\in\mathbb{D}$ and $t_n+1$ holomorphic maps
$\tau_{n1},\cdots,\tau_{nt_n},\tau_{n(t_n+1)}$ from $\mathbb{D}$
to ${\rm Teich}(W)$ with
$$\tau_{n1}(0)=R(\psi_0),\ \tau_{n1}(a_{n1})=R(h_{n1}),\
\tau_{nj}(0)=R(h_{n(j-1)})\ (j=2,\cdots,t_n),$$
$$\tau_{nj}(a_{nj})=R(h_{nj})\ (j=2,\cdots,t_n),\ \tau_{n(t_n+1)}(0)=R(h_{nt_n}),\
\tau_{n(t_n+1)}(a_{n(t_n+1)})=R(\psi_1),$$
and
$$\lim\limits_{n\to+\infty}\sum_{j=1}^{t_n+1}\ln\frac{1+|a_{nj}|}{1-|a_{nj}|}=d_T(f,g).$$
For all $n\geq1$, we let $H_{nj}=P\comp h_{nj}^{-1}\comp e^{2\pi i\alpha}$ for $j=1,\cdots,t_n$, $H_{n0}=f$
and $H_{n(t_n+1)}=g$.
Similar to the proof of Lemma \ref{L3.2}, it follows that
for all $n\geq1$ and $j=1,2,\cdots,t_n+1$,
there exist quasiconformal maps $\phi_{nj}$ from $\hat{\mathbb{C}}$ to itself with dilatation not exceeding $\frac{1+|a_{nj}|}{1-|a_{nj}|}$ such that
$\phi_{nj}(0)=0$, $\phi_{nj}(\infty)=\infty$, $\phi_{nj}(\mathcal{O}_{H_{n(j-1)}})=\mathcal{O}_{H_{nj}}$ and $\phi_{nj}\comp H_{n(j-1)}(z)=H_{nj}\comp\phi_{nj}(z)$ for $z\in\mathcal{O}_{H_{n(j-1)}}$. Then for all $n\geq1$, $\phi_n=\phi_{n(t_n+1)}\comp\cdots\comp\phi_{n2}\comp\phi_{n1}$ with dilatation not exceeding $\prod_{j=1}^{t_n+1}\frac{1+|a_{nj}|}{1-|a_{nj}|}$ satisfies
$\phi_n(0)=0$, $\phi_n(\infty)=\infty$, $\phi_n(\mathcal{O}_{f})=\mathcal{O}_{g}$ and $\phi_n\comp f(z)=g\comp\phi_n(z)$ for $z\in\mathcal{O}_f$.
Let $\phi$ be a limiting map for some subsequence of $\{\phi_n\}_{n=1}^{+\infty}$. Then one can obtain that
$\phi$ is a quasiconformal with $\phi(0)=0$, $\phi(\infty)=\infty$, $\phi(\mathcal{O}_{f})=\mathcal{O}_{g}$, $\phi\comp f(z)=g\comp\phi(z)$ for $z\in\mathcal{O}_f$ and dilatation not exceeding $e^{d_T(f,g)}$.
Thus $\phi$ is what we want.

If $\alpha$ is a Brjuno number, it follows from Lemma \ref{L3.2} that
$\phi(\Delta_f)=\Delta_g$. In this case, we can further request that $\phi$ is conformal on $\Delta_f$ and conjugates $f$ on $\Delta_f$ to $g$ on $\Delta_g$. In fact, let conformal maps $\psi_f:\mathbb{D}\to\Delta_f$ and $\psi_g:\mathbb{D}\to\Delta_g$ such that $\psi_f^{-1}\comp f\comp\psi_f(z)=e^{2\pi i\alpha}z$ and $\psi_g^{-1}\comp g\comp\psi_g(z)=e^{2\pi i\alpha}z$ for all $z\in\mathbb{D}$. According to \cite{C17} or \cite{SY}, both $\Delta_f$ and $\Delta_g$ are Jordan domains.
Thus there exists a homeomorphic extension
$\hat{\psi}_f${\rm(}or $\hat{\psi}_g${\rm)} of $\psi_f${\rm(}or $\psi_g${\rm)} from $\overline{\mathbb{D}}$ to $\overline{\Delta}_f${\rm(}or $\overline{\Delta}_g${\rm)} such that
$\hat{\psi}_f^{-1}\comp f\comp\hat{\psi}_f(z)=e^{2\pi i\alpha}z$
{\rm(}or $\hat{\psi}_g^{-1}\comp g\comp\hat{\psi}_g(z)=e^{2\pi i\alpha}z${\rm)} for all $z\in\overline{\mathbb{D}}$.
Then $\hat{\psi}_f$ conjugates $e^{2\pi i\alpha}z$ on $\partial\mathbb{D}$ to $f$ on $\partial\Delta_f$ and
$\hat{\psi}_g$ conjugates $e^{2\pi i\alpha}z$ on $\partial\mathbb{D}$ to $g$ on $\partial\Delta_g$.
It follows from Lemma \ref{L3.2} that $\phi$ conjugates $f$ on $\partial\Delta_f$ to $g$ on $\partial\Delta_g$.
Thus $\hat{\psi}_g^{-1}\comp\phi\comp\hat{\psi}_f$ conjugates $e^{2\pi i\alpha}z$ on $\partial\mathbb{D}$ to
$e^{2\pi i\alpha}z$ on $\partial\mathbb{D}$. This implies that $\hat{\psi}_g^{-1}\comp\phi\comp\hat{\psi}_f(z)=e^{2\pi i\theta}z$
for all $z\in\partial\mathbb{D}$ and some $\theta\in\mathbb{R}$. We define a map
\[\hat{\phi}(z)=\left\{\begin{matrix}\phi(z),&z\in\mathbb{C}\setminus\Delta_f\\
\psi_g(e^{2\pi i\theta}(\psi_f^{-1}(z))),&z\in\Delta_f\end{matrix}\right..\]
It follows easily from Rickman Theorem\upcite{R} that $\hat{\phi}$ is quasiconformal and hence $\hat{\phi}$
satisfies our requestness.
\end{proof}
Lemma \ref{L3.2} can help us prove the following lemma, which give a close connection
between Siegel disks of original maps and Siegel disks of renormalization maps.

\begin{lemma}
\label{p3.2}Assume $\alpha\in\mathbb{HT}_N$ is a Brjuno number and $h\in\mathcal{IS}_{\alpha}\cup\{Q_{\alpha}\}$. Then
$$\Delta_{\mathcal{R}(h)}={\rm Exp}(\Phi_h((\mathcal{C}_h\cup\mathcal{C}_h^{\#})\cap\Delta_h))\cup\{0\}.$$
\end{lemma}
\begin{proof}
Let $(\mathcal{C}_h\cup\mathcal{C}_h^{\#})/h$ be the natural Riemann surface with boundary obtained by gluing the `right' boundary and the `left' boundary of $\mathcal{C}_h\cup\mathcal{C}_h^{\#}$ by $h$.
Let $\pi_h:\mathcal{C}_h\cup\mathcal{C}_h^{\#}\to(\mathcal{C}_h\cup\mathcal{C}_h^{\#})/h$ is the natural projection.
We first have the following claim(See Appendix A for its proof).

\vspace{0.1cm}
\noindent{\bf Claim A:} The set $\pi_h((\mathcal{C}_h\cup\mathcal{C}_h^{\#})\cap\Delta_h)$ is a connected open subset of $(\mathcal{C}_h\cup\mathcal{C}_h^{\#})/h$ containing a neighborhood of the `0' end, not intersecting
the boundary of $(\mathcal{C}_h\cup\mathcal{C}_h^{\#})/h$ and
$$\partial(\pi_h((\mathcal{C}_h\cup\mathcal{C}_h^{\#})\cap\Delta_h))
=\pi_h((\mathcal{C}_h\cup\mathcal{C}_h^{\#})\cap\partial\Delta_h).$$

\vspace{0.1cm}
\noindent It is clear that ${\rm Exp}\comp\Phi_h$ induces an isomorphism $\Psi_h$ from $(\mathcal{C}_h\cup\mathcal{C}_h^{\#})/h$ to ${\rm Exp}(\Phi_h(\mathcal{C}_h\cup\mathcal{C}_h^{\#}))$,
a cylinder with boundary.
By the definitions of the renormalization and Siegel disk, we have
\[\Delta_{\mathcal{R}(h)}\setminus\{0\}\subset\mathcal{R}(h)(\Delta_{\mathcal{R}(h)}\setminus\{0\})
\subset{\rm Exp}(\Phi_h(\mathcal{C}_h\cup\mathcal{C}_h^{\#})).\]
Then $\Psi_h^{-1}(\Delta_{\mathcal{R}(h)}\setminus\{0\})$ is a connected open subset of
$(\mathcal{C}_h\cup\mathcal{C}_h^{\#})/h$ containing a neighborhood of the `0' end.
To prove the lemma, we only need to prove
$$\Psi_h^{-1}(\Delta_{\mathcal{R}(h)}\setminus\{0\})=\pi_h((\mathcal{C}_h\cup\mathcal{C}_h^{\#})\cap\Delta_h).$$
If
$\Psi_h^{-1}(\Delta_{\mathcal{R}(h)}\setminus\{0\})
\not\subset\pi_h((\mathcal{C}_h\cup\mathcal{C}_h^{\#})\cap\Delta_h),$
there exists a point
$$b'\in\Psi_h^{-1}(\Delta_{\mathcal{R}(h)}\setminus\{0\})
\cap\partial(\pi_h((\mathcal{C}_h\cup\mathcal{C}_h^{\#})\cap\Delta_h)),$$
since $\Psi_h^{-1}(\Delta_{\mathcal{R}(h)}\setminus\{0\})$ is connected and
$\Psi_h^{-1}(\Delta_{\mathcal{R}(h)}\setminus\{0\})\cap\pi_h((\mathcal{C}_h\cup\mathcal{C}_h^{\#})
\cap\Delta_h)\not=\emptyset.$
%Let $b\in\pi_h^{-1}(b')$.
It follows from Claim A that there exists
$$b\in\pi_h^{-1}(\Psi_h^{-1}(\Delta_{\mathcal{R}(h)}\setminus\{0\}))\cap\partial\Delta_h$$
such that $\pi_h(b)=b'$.
Then we can choose a neighborhood $U$ of $b$ such that ${\rm Exp}(\Phi_h(U))\subset\Delta_{\mathcal{R}(h)}$
and there exists a point $b''$ of $U$ such that the orbit of $b''$ under $h$ converges to $\infty$ for
$h=Q_{\alpha}$ or $b''$ can be only iterated finitely many times under $h$ for $h\in\mathcal{IS}_{\alpha}$.
This contradicts the following claim.

\vspace{0.1cm}
\noindent{\bf Claim B:} For any $z\in\mathcal{P}_h$ with ${\rm Exp}(\Phi_h(z))\in\Delta_{\mathcal{R}(h)}$,
infinite elements of the orbit of $z$ under $h$ are contained in $\mathcal{P}_h$.

\vspace{0.1cm}
\noindent In fact, since ${\rm Exp}(\Phi_h(z))\in\Delta_{\mathcal{R}(h)}$,
we have that $w_m=\mathcal{R}(h)^m({\rm Exp}(\Phi_h(z)))$ is well defined for all $m\geq1$.
By the definition of the renormalization, there exists $z_m\in\mathcal{P}_h$ such that ${\rm Exp}(\Phi_h(z_m))=w_m$ for all $m\geq1$.
By Lemma \ref{BCL1}, there exists a strictly increasing sequence of positive integers $\{l_m\}_{m=1}^{+\infty}$ such that $h^{l_m}(z)=z_m$ for all $m\geq1$. Thus $\{z_m\}_{m=1}^{+\infty}\subset\mathcal{P}_h$ is a subsequence of the
orbit of $z$ under $h$. This implies Claim B holds.

If
$$\pi_h((\mathcal{C}_h\cup\mathcal{C}_h^{\#})\cap\Delta_h)\not\subset
\Psi_h^{-1}(\Delta_{\mathcal{R}(h)}\setminus\{0\}),$$
It follows from Claim A that
$$\pi_h((\mathcal{C}_h\cup\mathcal{C}_h^{\#})\cap\Delta_h)\cap
\partial(\Psi_h^{-1}(\Delta_{\mathcal{R}(h)}\setminus\{0\}))\not=\emptyset.$$
Thus there exists a point $c\in(\mathcal{C}_h\cup\mathcal{C}_h^{\#})\cap\Delta_h$
such that ${\rm Exp}(\Phi_h(c))\in\partial\Delta_{\mathcal{R}(h)}$.
Then by Lemma \ref{L3.2} there exists a point $c'\in\Delta_h$ such that ${\rm Exp}(\Phi_h(c'))$ is
in the critical orbit of $\mathcal{R}(h)$. Since ${\rm Exp}(\Phi_h(cv_h))=cv_{\mathcal{R}(h)}$,
by Lemma \ref{BCL1} $h^m(cv_h)=c'$ for some nonnegative integer $m$.
This contradicts Lemma \ref{L3.2}.
\end{proof}

Assume that $f$ is a holomorphic function defined in the neighborhood of 0 with $f(0)=0$,
$f'(0)=e^{2\pi i\alpha} (\alpha\in\mathbb{R}/\mathbb{Z})$ and a Siegel disk $\Delta_f$ at 0.
Let $\mathbb{D}_r=\{z\in\mathbb{C}||z|<r\}$ for all $0<r<+\infty$.
If $\phi:\mathbb{D}\to\Delta_f$ be a conformal map such that
$\phi(0)=0$ and $f(\phi(z))=\phi(e^{2\pi i\alpha}z)$ for $z\in\mathbb{D}$, then for all $0<r<1$,
we call $\phi(\mathbb{D}_r)$ the r-disk of the Siegel disk $\Delta_f$. We denote by $\Delta_f(r)$
the r-disk of the Siegel disk $\Delta_f$.

\begin{corollary}
\label{c3.1}Assume that $\alpha\in\mathbb{HT}_N$ is a Brjuno number and $h\in\mathcal{IS}_{\alpha}\cup\{Q_{\alpha}\}$. Then
for all $0<r<1$, there exist a positive integer $K$ and two real numbers $0<r',r''<1$ such that
\begin{equation}
\label{c2}\Delta_h(r)=\cup_{j=0}^{K}h^j((\mathcal{C}_h\cup\mathcal{C}_h^{\#})\cap\Delta_h(r))\cup\{0\},
\end{equation}
\begin{equation}
\label{c3}\Delta_{\mathcal{R}(h)}(r')=
{\rm Exp}\comp\Phi_h((\mathcal{C}_h\cup\mathcal{C}_h^{\#})\cap\Delta_h(r))\cup\{0\}
\end{equation}
and
\begin{equation}
\label{c4}\Delta_{\mathcal{R}(h)}(r)=
{\rm Exp}\comp\Phi_h((\mathcal{C}_h\cup\mathcal{C}_h^{\#})\cap\Delta_h(r''))\cup\{0\}.
\end{equation}
\end{corollary}

\begin{proof}
Let the notations $(\mathcal{C}_h\cup\mathcal{C}_h^{\#})/h$, $\Psi_h$, $\pi_h$
be the same as those in the proof of Lemma \ref{p3.2}. For convenience, let $C=(\mathcal{C}_h\cup\mathcal{C}_h^{\#})\cap\Delta_h$.
Claim A in the proof of Lemma \ref{p3.2} tells us that $\pi_h(C)$
is a connected open subset of $\pi_h(\mathcal{C}_h\cup\mathcal{C}_h^{\#})$
containing a neighborhood of the `0' end, not intersecting
the boundary of $(\mathcal{C}_h\cup\mathcal{C}_h^{\#})/h$ and
$$\partial(\pi_h(C))
=\pi_h((\mathcal{C}_h\cup\mathcal{C}_h^{\#})\cap\partial\Delta_h).$$
Let $\phi_h$ be a conformal map from $\mathbb{D}$ onto $\Delta_h$ such that
$\phi_h(0)=0$ and $h(\phi_h(z))=\phi_h(e^{2\pi i\alpha}z)$ for $z\in\mathbb{D}$.
We define a map
\[\rho:\pi_h(C)\to\mathbb{R}^+,\
z\mapsto|\phi_h^{-1}(\pi_h^{-1}(z))|.\]
Following Appendix A, the map $\rho$ is well defined.
It is easy to check that

\vspace{0.1cm}
\noindent(1) $\rho$ is continuous and open;\\
(2) $\rho(\pi_h(C))$
is connected;\\
(3) $\rho(z)\to0$ as $z\to$ the `0' end and $\rho(z)\to1$ as $z\to\partial(\pi_h(C))$, say, the boundary of $\pi_h(C)$.

\vspace{0.1cm}
\noindent From (1), (2) and (3), we have
$\rho(\pi_h(C))=]0,1[$.
This implies that for all $0<r<1$, there exists a nonempty open subset $I_r$ of
$\{z\in\mathbb{D}:|z|=r\}$, contained in $\phi_h^{-1}(C)$. Moreover, fix $0<r<1$,
we can further request that
there exists a positive number $\delta$ depending on $r$ such that for any $0<r_1\leq r$,
the length of $I_{r_1}$ is not less than $\delta$. Thus (\ref{c2}) holds.
From the proof of Lemma \ref{p3.2}, we know that the map $\Psi_h$ induces an isomorphism from $\pi_h(C)$ to $\Delta_{\mathcal{R}(h)}\setminus\{0\}$.
For all $0<r<1$, set
$$E_r:=\pi_h((\mathcal{C}_h\cup\mathcal{C}_h^{\#})\cap\Delta_h(r)).$$
We first prove that for all $0<r<1$, there exists $0<r'<1$ such that
\begin{equation}
\label{f1}\Psi_h^{-1}(\partial\Delta_{\mathcal{R}(h)}(r))\subset\partial E_{r'}.
\end{equation}
Indeed, given $w\in\partial\Delta_{\mathcal{R}(h)}(r)$, by Lemma \ref{p3.2} there exists
$z\in(\mathcal{C}_h\cup\mathcal{C}_h^{\#})\cap\Delta_h$ such that $\pi_h(z)=\Psi_h^{-1}(w)$.
Then there exists $0<r'<1$ such that $\phi_h^{-1}(z)\in\{z\in\mathbb{D}:|z|=r'\}$.
Thus $\Psi_h^{-1}(w)\in\partial E_{r'}$.
For any $w'\in\partial\Delta_{\mathcal{R}(h)}(r)$, there exists a sequence $\{k_j\}_{j=1}^{+\infty}$
of positive integers such that
\begin{itemize}
\item $k_j\to+\infty$ as $j\to+\infty$,
\item $\mathcal{R}(h)^{k_j}(w')\to w$ as $j\to+\infty$.
\end{itemize}
Then by Lemma \ref{p3.2}, there exist $z'_j\in(\mathcal{C}_h\cup\mathcal{C}_h^{\#})\cap\Delta_h$ $(j=0,1,2,\cdots)$ such that
\begin{itemize}
\item $\pi_h(z'_0)=\Psi_h^{-1}(w')$ and $\pi_h(z'_j)=\Psi_h^{-1}(\mathcal{R}(h)^{k_j}(w'))$($j\geq1$);
\item $z'_j\to z$ as $j\to+\infty$.
\end{itemize}
By Lemma \ref{BCL1}, there exists a sequence $\{t_j\}_{j=1}^{+\infty}$ of positive integers such that
\begin{itemize}
\item $t_j\to+\infty$ as $j\to+\infty$,
\item $h^{t_j}(z'_0)=z'_j$($j\geq1$).
\end{itemize}
Thus there exists $0<r''<1$ such that $\phi_h^{-1}(z'_j)\in\{z\in\mathbb{D}:|z|=r''\}$($j\geq0$).
Since $z'_j\to z$ as $j\to+\infty$, we have $r''=r'$ and hence
$\Psi_h^{-1}(w')=\pi_h(z'_0)\in\partial E_{r'}$. The relation (\ref{f1}) follows from arbitrariness of $w'$.
We define a map
$$\tau:]0,1[\to]0,1[,\ r\mapsto r'.$$
According to the properties of the map $\rho$,
it is easy to check that $\tau$ is a strictly increasing, surjective, continuous map.
This implies
$$\Psi_h^{-1}(\partial\Delta_{\mathcal{R}(h)}(r))=\partial E_{\tau(r)}.$$
Thus
\begin{align*}
\Delta_{\mathcal{R}(h)}(r)\setminus\{0\}&=\cup_{0<l\leq r}\partial\Delta_{\mathcal{R}(h)}(l)
=\cup_{0<l\leq r}\Psi_h(\partial E_{\tau(l)})\\
&=\Psi_h(E_{\tau(r)})=
{\rm Exp}\comp\Phi_h((\mathcal{C}_h\cup\mathcal{C}_h^{\#})\cap\Delta_h(\tau(r))).
\end{align*}
This formula implies that (\ref{c3}) and (\ref{c4}) hold.
\end{proof}

\section{The proof of Proposition \ref{p139}}
Let $P_{\theta}(z)=e^{2\pi i\theta}z+z^2$, $\theta\in B=\{z:|z-\frac{1}{2}|<\frac{1}{2}\}\subset\mathbb{C}$ and
$D_R=\{z:|z|<R\}$. Let $R$ be large enough so that for all $\theta\in B$, $D_R$ contains the finite critical value
of $P_{\theta}$ and $P_{\theta}^{-1}(D_R)$ is compactly contained in $D_R$. Then
$P_{\theta}^{-1}(D_R)$ only has a simply connected component $U_{\theta,R}$. Thus
$P_{\theta}|_{U_{\theta,R}}:U_{\theta,R}\to D_R$
is a quadratic-like mapping. Similarly, we can choose $R'>R$
such that $P_{\theta}^{-1}(D_{R'})$ only has a simply connected component $U_{\theta,R'}$, compactly contained in
$D_R$. Then $P_{\theta}|_{U_{\theta,R'}}:U_{\theta,R'}\to D_{R'}$ is a quadratic-like mapping.

Since $f_{\theta,\epsilon,d}(z)=P_{\theta}(z)+\epsilon z^d$, we can assume $\epsilon$ is small enough so that
for all $\theta\in B$, we have
$f_{\theta,\epsilon,d}(\partial U_{\theta,R'})\cap D_R=\emptyset$,
every element of $D_R$ has two preimages(counting multiplicity) under $f_{\theta,\epsilon,d}|_{U_{\theta,R'}}$
and $D_R$ contains a critical value of $f_{\theta,\epsilon,d}|_{U_{\theta,R'}}$.
Then for all $\theta\in B$,
$$V_{\theta}=(f_{\theta,\epsilon,d}|_{U_{\theta,R'}})^{-1}(D_R)=(f_{\theta,\epsilon,d})^{-1}(D_R)\cap
U_{\theta,R'}$$
is a simply connected region and is compactly contained in $D_{R}$.
Thus for all $\theta\in B$, $f_{\theta,\epsilon,d}: V_{\theta}\to D_{R}$
is a quadratic-like mapping. Furthermore,
$f_{\theta,\epsilon,d}(z) (\theta\in B,z\in V_{\theta})$ is an analytic family of quadratic-like mappings.
Then according to \cite{DH}, there exits a continuous map $\chi:B\to\mathbb{C}$ such that
$f_{\theta,\epsilon,d}|_{V_{\theta}}$ is hybrid equivalent to $g_{\theta}(z)=z^2+\chi(\theta)$ with
the conjugate map $\phi_{\theta}$.
Since $f_{\theta,\epsilon,d}|_{V_{\theta}}$ has two simple fixed points for $\theta\in B$,
$g_{\theta}$ has also two simple fixed points for $\theta\in B$.
Then the two fixed points $y_1(\theta)$, $y_2(\theta)$ of $g_{\theta}$ depends continuously on $\theta$.
Thus both the multiplier of $g_{\theta}$ at $y_1(\theta)$ and
the multiplier of $g_{\theta}$ at $y_2(\theta)$ depend continuously on $\theta$.
(According to \cite{DH}, $\phi_{\theta}$ doesn't necessarily depend continuously on $\theta$.
Thus the image $\phi_{\theta}(0)$ doesn't necessarily depend continuously on $\theta$.)
Denote by $\mathcal{B}$ the set consisting of all Brjuno numbers.
For all $\theta\in\mathcal{B}$, $0$ is a Siegel fixed point of $f_{\theta,\epsilon,d}|_{V_{\theta}}$.
Thus the multiplier of $g_{\theta}$ at $\phi_{\theta}(0)$ is the same as that of $f_{\theta,\epsilon,d}|_{V_{\theta}}$ at $0$.
Consider two sets
\[E_1=\{\theta\in B\cap\mathcal{B}:\phi_{\theta}(0)=y_1(\theta)\}\ {\rm and}\
E_2=\{\theta\in B\cap\mathcal{B}:\phi_{\theta}(0)=y_2(\theta)\}.\]
It is easy to see that at most one of $E_1$ and $E_2$ is nowhere dense.
Without loss of generality, we let $E_1$ be not nowhere dense.
Thus the closure of $E_1$ contains some open interval $I$ of  $\mathbb{R}\cap B$.
According to the continuous dependence of the multiplier of $g_{\theta}$ at $y_1(\theta)$ on $\theta$,
the multipliers of $g_{\theta}$ at $y_1(\theta)$, as $\theta$ changes in $B$, can take a certain open arc (corresponding to $I$) of the unit circle.
Then for all $\theta\in I$, the multiplier of $g_{\theta}$ at $y_1(\theta)$ is $e^{2\pi i\theta}$
and hence $g_{\theta}$ is affine conjugate to $P_{\theta}$.
By Theorem A, we have
$I\cap\mathcal{S}\not=\emptyset$. Then there exists
$\theta\in I\cap\mathcal{S}$ such that
$f_{\theta,\epsilon,d}|_{V_{\theta}}$ is hybrid equivalent to $g_{\theta}$.
Thus there exists $\theta\in \mathcal{S}$ such that
$f_{\theta,\epsilon,d}|_{V_{\theta}}$ is hybrid equivalent to $P_{\theta}$.
Note that we can replace $\mathcal{S}$ by $\mathcal{S}_1$ (or $\mathcal{S}_2$).
Thus there exists $\theta\in\mathcal{S}_1$ (or $\theta\in\mathcal{S}_2$) such that
$f_{\theta,\epsilon,d}|_{V_{\theta}}$ is hybrid equivalent to $P_{\theta}$.

\section{The proof of Proposition \ref{p239}}
Suppose that there exists $\theta\in\mathbb{R}/\mathbb{Z}$ such that $f_{\theta}$ is renormalizable, that is,
there exist two Jordan regions $U,V\subset\mathbb{C}$ such that $U$ is compactly contained in $V$ and
there exists a least integer $p\geq1$ such that $f_{\theta}^p|_U:U\to V$ is a quadratic-like mapping with
connected Julia set. Then $f_{\theta}^p|_U$ only has a critical point $c$ in $U$ and $U$ contains the postcritical
set $\mathcal{O}_{f_{\theta}^p|_U}(c)$ of $c$. Assume that $0\in U$. Then $0\in V$ and hence
$(f_{\theta}^p|_U)^{-1}(0)$ has exactly two elements (counting multiplicity). Since
$f_{\theta}^{-1}(0)=\{0,-1\}$ and $-1$ is a critical point of $f_{\theta}$,
all elements in $(f_{\theta}^p|_U)^{-1}(0)$, except $0$, are critical points of $f_{\theta}^p|_U$.
Thus $(f_{\theta}^p|_U)^{-1}(0)$ has at least three elements (counting multiplicity), and
this is a contradiction. Thus $0\not\in U$. To complete the proof, we only need to prove $0\in U$.

We first observe that $f_{\theta}$ has two finite critical points $c_0=-\frac{1}{3}$ and $c_1=-1$
with the corresponding critical values $-\frac{4}{27}e^{2\pi i\theta}$ and $0$.
Thus we have
$$\bigcup\limits_{n\geq1}(f_{\theta}^p|_U)^n(c)\subset\bigcup\limits_{n\geq1}f_{\theta}^n(c_0)\
{\rm or}\
\bigcup\limits_{n\geq1}(f_{\theta}^p|_U)^n(c)\subset\bigcup\limits_{n\geq1}f_{\theta}^n(c_1).$$
If
$$\bigcup\limits_{n\geq1}(f_{\theta}^p|_U)^n(c)\subset\bigcup\limits_{n\geq1}f_{\theta}^n(c_1),$$
then $\bigcup\limits_{n\geq1}(f_{\theta}^p|_U)^n(c)=\{0\}$ and hence $0=f_{\theta}^p|_U(c)\in U$.
If
$$\bigcup\limits_{n\geq1}(f_{\theta}^p|_U)^n(c)\subset\bigcup\limits_{n\geq1}f_{\theta}^n(c_0),$$
then when $\theta$ is a rational number, $0$ is a parabolic fixed point of $f_{\theta}$. So
$f_{\theta}^n(c_0)\xrightarrow[n\to+\infty]{}0$. This implies
$(f_{\theta}^p|_U)^n(c)\xrightarrow[n\to+\infty]{}0$. Thus $0\in\mathcal{O}_{f_{\theta}^p|_U}(c)\subset U$.
When $\theta$ is a irrational number, $0$ is a Cremer point or a Siegel point of $f_{\theta}$. Then
$0$ is also a Cremer point or a Siegel point of $f_{\theta}^p$. According to [Theorem 1.3,\cite{ST00}],
there exists a recurrent critical point $c'$ of $f_{\theta}^p$ such that
when $0$ is a Cremer point of $f_{\theta}$,
$0$ is contained in the postcritical set of $c'$ under $f_{\theta}^p$, that is,
$0\in\mathcal{O}_{f_{\theta}^p}(c')$;
when $0$ is a Siegel point of $f_{\theta}$,
$\partial\Delta_{\theta}\subset\mathcal{O}_{f_{\theta}^p}(c')$,
where $\Delta_{\theta}$ is the Siegel disk of $f_{\theta}^p$ at $0$.
It is easy to see that there exists $0\leq k\leq p-1$ such that
$f_{\theta}^k(\mathcal{O}_{f_{\theta}^p}(c'))\subset\mathcal{O}_{f_{\theta}^p|_U}(c)$.
Then when $0$ is a Cremer point of $f_{\theta}$,
$0=f_{\theta}^k(0)\in f_{\theta}^k(\mathcal{O}_{f_{\theta}^p}(c'))\subset\mathcal{O}_{f_{\theta}^p|_U}(c)\subset U$;
when $0$ is a Siegel point of $f_{\theta}$,
$\partial\Delta_{\theta}=f_{\theta}^k(\partial\Delta)\subset f_{\theta}^k(\mathcal{O}_{f_{\theta}^p}(c'))\subset\mathcal{O}_{f_{\theta}^p|_U}(c)\subset U$. Then
\[\mathbb{C}\setminus\overline{U}=
(\mathbb{C}\setminus\overline{U})\cap(\Delta_{\theta}\cup\partial\Delta_{\theta}\cup
(\mathbb{C}\setminus\overline{\Delta}_{\theta}))=((\mathbb{C}\setminus\overline{U})\cap\Delta_{\theta})
\cup((\mathbb{C}\setminus\overline{U})\cap(\mathbb{C}\setminus\overline{\Delta}_{\theta}))
.\]
It follows from the connectedness of $\mathbb{C}\setminus\overline{U}$ that
$(\mathbb{C}\setminus\overline{U})\cap\Delta_{\theta}=\emptyset$ or $(\mathbb{C}\setminus\overline{U})\cap(\mathbb{C}\setminus\overline{\Delta}_{\theta})=\emptyset$.
Since $\Delta_{\theta}$ is bounded,
$(\mathbb{C}\setminus\overline{U})\cap(\mathbb{C}\setminus\overline{\Delta}_{\theta})\not=\emptyset$.
Then $(\mathbb{C}\setminus\overline{U})\cap\Delta_{\theta}=\emptyset$ and hence
$\Delta_{\theta}\subset\overline{U}$. Since $\Delta_{\theta}$ is an open set,
$\Delta_{\theta}\subset U$. Thus $0\in\Delta_{\theta}\subset U$.

\section{The proof of Theorem \ref{T20}}
In this section, we will prove Theorem \ref{T20}. Before doing it, we give three lemmas.

\begin{lemma}[Buff and Ch\'eritat]
\label{T1.1s}Assume $\alpha:=[a_0,a_1,a_2,\cdots]$ and $\theta:=[0,t_1,t_2,\cdots]$ are Brjuno numbers and let $p_n/q_n$ be the approximants to $\alpha$. Assume
\[\alpha_n:=[a_0,a_1,\cdots,a_n,A_n,t_1,t_2,\cdots]\]
with $(A_n)$ a sequence of positive integers such that
\begin{equation*}
\limsup\limits_{n\to+\infty}\sqrt[q_n]{\log A_n}\leq1.
\end{equation*}
For any $0<r<1$, let $\Delta(r)$ be the r-disk of the Siegel disk of $P_{\alpha}$ and
let $\Delta_n'(r)$ be the Siegel disk of the restriction of $P_{\alpha_n}$ to $\Delta(r)$. For any nonempty open set $U\subset\Delta(r)$,
\[\liminf\limits_{n\to+\infty}{\rm dens}_{U}(\Delta_n'(r))\geq\frac{1}{2}.\]
\end{lemma}

\begin{remark}
{\rm It is easy to see that Lemma \ref{T1.1s} implies Theorem B, and from the proof of
[Proposition 6, \cite{BC}], one obtains Lemma \ref{T1.1s} easily.}
\end{remark}

\begin{lemma}
\label{L4.1}Let $\Omega_1$, $\Omega_2$ be two bounded nonempty connected open sets, $\{E_n\}_{n=1}^{+\infty}$ is a sequence of open subsets of $\Omega_1$ and $\phi:\Omega_1\to\Omega_2$ is a quasiconformal map. If there exists a positive number c such that for any nonempty open set $U\subset\Omega_1$,
\[\liminf\limits_{n\to+\infty}{\rm dens}_{U}(E_n)\geq c,\]
then for any nonempty open set
$U\subset\Omega_2$,
\[\liminf\limits_{n\to+\infty}{\rm dens}_{U}(\phi(E_n))\geq c.\]
\end{lemma}
\begin{proof}
According to Jacobian Formula of quasiconformal maps \cite{HA}, we have that for any nonempty open set $U\subset\Omega_2$,
\[\int_{\phi^{-1}(U)}J_{\phi}(z)dxdy={\rm area}(U)\]
and
\[\int_{\phi^{-1}(U)\cap E_n}J_{\phi}(z)dxdy
={\rm area}(U\cap\phi(E_n)).\]
Let $\{T_k\}_{k=1}^{+\infty}$ be a sequence of simple real value functions with $T_k(z)=\sum_{l=1}^{t_k}a_l\mathcal{X}_{E_{kl}}$,
where $a_1,\cdots,a_{t_k}$ are real numbers, $E_{k1},E_{k2},\cdots,E_{kt_k}$ are pairwise disjoint measurable sets with $\cup_{j=1}^{t_k}E_{kj}=\phi^{-1}(U)$ and $\mathcal{X}_{E_{kl}}$($l=1,\cdots,t_k$) are characteristic functions of $E_{kl}$($l=1,\cdots,t_k$),
such that
\[\lim\limits_{k\to+\infty}\int_{\phi^{-1}(U)}|J_{\phi}(z)-T_k(z)|dxdy=0.\]
Thus to complete the proof of the lemma, we only need to prove for large $k$,
\[\liminf\limits_{n\to+\infty}\frac{\int_{\phi^{-1}(U)\cap E_n}T_k(z)dxdy}
{\int_{\phi^{-1}(U)}T_k(z)dxdy}\geq c.\]
To prove the above formula, it suffices to prove that for all $k$,
\[\liminf\limits_{n\to+\infty}\int_{\phi^{-1}(U)\cap E_n}\mathcal{X}_{E_{kj}}(z)dxdy
\geq c\int_{\phi^{-1}(U)}\mathcal{X}_{E_{kj}}(z)dxdy\]
for $j=1,2,\cdots,t_k$.
We choose a sequence $\{V_q\}_{q=1}^{+\infty}$ of open subsets of $\phi^{-1}(U)$ such that
$$V_q\supset E_{kj}\cap\phi^{-1}(U)(q\geq1)\ {\rm and}\
\lim\limits_{q\to+\infty}{\rm area}(V_q\setminus (E_{kj}\cap\phi^{-1}(U)))=0.$$
For any $\epsilon>0$, there exists $K_0>0$ such that for all $q>K_0$,
\begin{align*}
\liminf\limits_{n\to+\infty}\int_{\phi^{-1}(U)\cap E_n}\mathcal{X}_{E_{kj}}(z)dxdy
&=\liminf\limits_{n\to+\infty}{\rm area}(\phi^{-1}(U)\cap E_{kj}\cap E_n)\\
&\geq\liminf\limits_{n\to+\infty}{\rm area}(V_q\cap E_n)-\epsilon\\
&\geq c\cdot{\rm area}(V_q)-\epsilon\\
&\geq c\int_{\phi^{-1}(U)}\mathcal{X}_{E_{kj}}(z)dxdy-\epsilon.
\end{align*}
This gives us
\[\liminf\limits_{n\to+\infty}\int_{\phi^{-1}(U)\cap E_n}\mathcal{X}_{E_{kj}}(z)dxdy
\geq c\int_{\phi^{-1}(U)}\mathcal{X}_{E_{kj}}(z)dxdy-\epsilon.\]
Let $\epsilon\to0$, the proof is completed.
\end{proof}

\begin{lemma}
\label{L4.2}Let $\alpha$ and $\alpha_n$ $(n=1,2,\cdots)$ be Brjuno numbers such that $\alpha_n\to\alpha$ as $n\to+\infty$. Suppose that $f_{\alpha}=h\comp e^{2\pi i\alpha}$ and $f_{\alpha_n}=h\comp e^{2\pi i\alpha_n}$
for some $h\in\mathcal{IS}_0$ or $f_{\alpha}=Q_{\alpha}$ and $f_{\alpha_n}=Q_{\alpha_n}$.
For all $0<r<1$ and $m\geq1$, the notation $\Delta_{f_{\alpha}}(r)${\rm(}or
$\Delta_{\mathcal{R}^m(f_{\alpha})}(r)${\rm)} is the r-disk of $\Delta_{f_{\alpha}}${\rm(}or $\Delta_{\mathcal{R}^m(f_{\alpha})}${\rm)} and
$\Delta'_{f_{\alpha_n}}(r)${\rm(}or $\Delta_{\mathcal{R}^m(f_{\alpha})}'(r)${\rm)} is
the Siegel disk of the restriction of $f_{\alpha_n}${\rm(}or $\mathcal{R}^m(f_{\alpha_n})${\rm)} to $\Delta_{f_{\alpha}}(r)${\rm(}or $\Delta_{\mathcal{R}^m(f_{\alpha})}(r)${\rm)} centering at 0.
Then for all $m\geq1$ and positive number $c$, the following two conditions are equivalent:

\vspace{0.1cm}
\noindent{\rm(a)} for all $0<r<1$ and any nonempty open set $U\subset\Delta_{f_{\alpha}}(r)$,
\[\liminf\limits_{n\to+\infty}{\rm dens}_U\Delta'_{f_{\alpha_n}}(r)\geq c;\]
{\rm(b)} for all $0<r<1$ and any nonempty open set $U\subset\Delta_{\mathcal{R}^m(f_{\alpha})}(r)$,
\[\liminf\limits_{n\to+\infty}{\rm dens}_U\Delta'_{\mathcal{R}^m(f_{\alpha_n})}(r)\geq c.\]
\end{lemma}
\begin{proof}
We denote the condition (b) with $m=1$ by ($b'$).
Then it is sufficient for proving the lemma to prove that (a) is equivalent to ($b'$).

Let
$C:=(\mathcal{C}_{f_{\alpha}}\cup\mathcal{C}_{f_{\alpha}}^{\#})\cap\Delta_{f_{\alpha}}$
and
$C_n:=(\mathcal{C}_{f_{\alpha_n}}\cup\mathcal{C}_{f_{\alpha_n}}^{\#})\cap\Delta_{f_{\alpha_n}}$
for all $n\geq1$. For all $0<r<1$, we let
$C(r):=(\mathcal{C}_{f_{\alpha}}\cup\mathcal{C}_{f_{\alpha}}^{\#})\cap\Delta_{f_{\alpha}}(r)$
and
$C_n'(r):=(\mathcal{C}_{f_{\alpha_n}}\cup\mathcal{C}_{f_{\alpha_n}}^{\#})\cap\Delta'_{f_{\alpha_n}}(r)$.
Next, we will prove that ($b'$) is equivalent to the condition ($a'$):
for all $0<r<1$ and any nonempty open set $U\subset C(r)$,
\[\liminf\limits_{n\to+\infty}{\rm dens}_UC_n'(r)\geq c.\]

\vspace{0.1cm}
\noindent($a'$)$\Rightarrow$($b'$): To prove ($b'$), we only need to prove that
for all $0<r'<r<1$ and any nonempty open set $U\subset\Delta_{\mathcal{R}(f_{\alpha})}(r')$,
\[\liminf\limits_{n\to+\infty}{\rm dens}_U\Delta'_{\mathcal{R}(f_{\alpha_n})}(r)\geq c.\]
By Corollary \ref{c3.1} , there exists $0<r''<1$ such that
\[\Delta_{\mathcal{R}(f_{\alpha})}(r')=
{\rm Exp}\comp\Phi_{f_{\alpha}}(C(r''))\cup\{0\}.\]
According to ($a'$), for any nonempty open set $U\subset C(r'')$,
\[\liminf\limits_{n\to+\infty}{\rm dens}_UC_n'(r'')\geq c.\]
Thus by the continuity of horn maps,
for any nonempty open set $U\subset\Delta_{\mathcal{R}(f_{\alpha})}(r')$,
\[\liminf\limits_{n\to+\infty}{\rm dens}_U{\rm Exp}\comp\Phi_{f_{\alpha_n}}(C_n'(r''))\geq c.\]
According to Lemma \ref{p3.2}, Corollary \ref{c3.1} and the continuity of horn maps,
we have that for large enough $n$,
$${\rm Exp}\comp\Phi_{f_{\alpha_n}}(C_n'(r''))\subset\Delta'_{\mathcal{R}(f_{\alpha_n})}(r).$$
Thus for any nonempty open set $U\subset\Delta_{\mathcal{R}(f_{\alpha})}(r')$,
\[\liminf\limits_{n\to+\infty}{\rm dens}_U\Delta'_{\mathcal{R}(f_{\alpha_n})}(r)\geq c.\]

\vspace{0.1cm}
Similarly, we can obtain ($b'$)$\Rightarrow$($a'$). At last, to complete the proof of this lemma,
we only need to prove that ($a'$) is equivalent to ($a$).
It is easy to see that ($a$) implies ($a'$). Next, we prove that ($a'$) implies ($a$).
From ($a'$), one can obtain that for all $0<r<1$ and any nonempty open set
$U_j\subset f_{\alpha}^j(C(r))$($j=0,1,\cdots$),
\begin{equation}
\label{f2}\liminf\limits_{n\to+\infty}{\rm dens}_{U_j}\Delta'_{f_{\alpha_n}}(r)\geq c
\end{equation}
for $j=0,1,\cdots$.
By Corollary \ref{c3.1}, there exists a positive integer $K$ such that
\[\Delta_{f_{\alpha}}(r)=\cup_{j=0}^Kf_{\alpha}^j(C(r)).\]
Thus for all $0<r<1$ and any nonempty open set $U\subset\Delta_{f_{\alpha}}(r)$,
there exist pairwise disjoint sets $E_j$($j=0,1,\cdots,K$) such that
$$E_j\subset f_{\alpha}^j(C(r))(j=0,1,\cdots,K),$$
$$U=\cup_{j=0}^KE_j$$
and
$${\rm area}(E_j\setminus\mathring{E_j})=0(j=0,1,\cdots,K).$$
Then by (\ref{f2})
\begin{equation*}
\liminf\limits_{n\to+\infty}{\rm dens}_U\Delta'_{f_{\alpha_n}}(r)=
\liminf\limits_{n\to+\infty}{\rm dens}_{\cup_{j=0}^K\mathring{E_j}}\Delta'_{f_{\alpha_n}}(r)\geq c.
\end{equation*}

\end{proof}

\begin{proof}[{\bf Proof of Theorem \ref{T20}}]
It is easy to see that to prove Theorem \ref{T20}, we only need to prove that
for all $0<r<1$ and any nonempty open set $U\subset\Delta_{f_{\alpha}}(r)$,
we have
\begin{equation}
\label{bbf1bb}\liminf\limits_{n\to+\infty}{\rm dens}_U(\Delta'_{f_{\alpha_n}}(r))\geq\frac{1}{2},
\end{equation}
where $\Delta_{f_{\alpha}}(r)$ is the r-disk of the Siegel disk of $f_{\alpha}=e^{2\pi i\alpha}h$
centering at the origin and $\Delta'_{f_{\alpha_n}}(r)$ is the Siegel disk of the restriction of
$f_{\alpha_n}=e^{2\pi i\alpha_n}h$ to $\Delta_{f_{\alpha}}(r)$ centering at the origin.

Let $g=h\comp e^{2\pi i\alpha}$ and $g_n=h\comp e^{2\pi i\alpha_n}$.
Then $g, g_n(n\geq1)\in\mathcal{IS}_N$ and it is easy to see that
to prove (\ref{bbf1bb}), we only need to prove that for any $0<r<1$, we have that
for any nonempty open set $U\subset\Delta_g(r)$,
\[\lim\limits_{n\to+\infty}{\rm dens}_U(\Delta'_{g_n}(r))\geq\frac{1}{2},\]
where $\Delta_g(r)$ is the r-disk of the Siegel disk of $g$ centering at 0
and $\Delta'_{g_n}(r)$ is the Siegel disk of the restriction of $g_n$
to $\Delta_g(r)$ centering at 0.
Next let us prove this. Let
\[\beta=[0,N,a_1,a_2,\cdots]\]
and
\[\beta_n=[0,N,a_1,a_2,\cdots,a_n,A_n,t_1,\cdots](n\geq1).\]
Then
$\mathcal{R}(Q_{\beta})\in\mathcal{IS}_{\alpha}$ and
$\mathcal{R}(Q_{\beta_n})\in\mathcal{IS}_{\alpha_n}$($n\geq1$).
Thus there exist $\tilde{\phi}_n\in\mathcal{S}^{qc}$, $n=1,2,\cdots$ such that $\mathcal{R}(Q_{\beta_n})=P\comp(\tilde{\phi}_n)^{-1}\comp e^{2\pi i\alpha_n}$, $n=1,2,\cdots$.
Since $\mathbb{HT}_N\ni\beta_n\to\beta\in\mathbb{HT}_N$ as $n\to+\infty$, by the continuity of horn maps,
for all $m\geq1$, the renormalization $\mathcal{R}^m(Q_{\beta_n})$ converges uniformly to $\mathcal{R}^m(Q_{\beta})$ on $V_{\mathcal{R}^m(Q_{\beta})}$.
According to [Theorem 6.3 ,\cite{IS}] and
the property of the renormalization $\mathcal{R}$, we have that $\{R(\tilde{\phi}_n)\}_{n=1}^{+\infty}$ is bounded in ${\rm Teich}(W)$ and hence there exists $M>0$ such that
\[d_T(\mathcal{R}(Q_{\beta_n}),g_n)<M\]
for all $n\geq1$.
By Theorem C, we have
\[d_T(\mathcal{R}^{m+1}(Q_{\beta_n}),\mathcal{R}^m(g_n))<\lambda^mM\]
for all $n\geq1$ and $m\geq0$. Fix $m\geq0$, by Lemma \ref{L4.3}, for all $n$, there exists a quasiconformal map $\phi_n$ from $\hat{\mathbb{C}}$ onto itself with dilatation not exceeding $e^{\lambda^mM}$ such that
\[\phi_n(\mathcal{O}_{\mathcal{R}^{m+1}(Q_{\beta_n})})=\mathcal{O}_{\mathcal{R}^m(g_n)}\ {\rm and}\
\phi_n(\Delta_{\mathcal{R}^{m+1}(Q_{\beta_n})})=\Delta_{\mathcal{R}^m(g_n)},\]
and
\[\phi_n\comp\mathcal{R}^{m+1}(Q_{\beta_n})(z)=\mathcal{R}^m(g_n)\comp\phi_n(z)\
{\rm for}\ z\in\mathcal{O}_{\mathcal{R}^{m+1}(Q_{\beta_n})}\cup\Delta_{\mathcal{R}^{m+1}(Q_{\beta_n})}.\]
Since $\mathcal{R}^m(g_n)$(or $\mathcal{R}^{m+1}(Q_{\beta_n})$) converges uniformly to
$\mathcal{R}^m(g)$(or $\mathcal{R}^{m+1}(Q_{\beta})$) on
$V_{\mathcal{R}^m(g)}$(or $V_{\mathcal{R}^{m+1}(Q_{\beta})}$) as $n\to+\infty$,
the critical orbit of $\mathcal{R}^m(g_n)$(or $\mathcal{R}^{m+1}(Q_{\beta_n})$) converges pointwise to that of
$\mathcal{R}^m(g)$(or $\mathcal{R}^{m+1}(Q_{\beta})$) as $n\to+\infty$. Thus
$\{\phi_n\}_{n=1}^{\infty}$ is precompact for the topology of uniform convergence and every limit $\phi$ is
a quasiconformal map from $\hat{\mathbb{C}}$ onto itself with dilatation not exceeding $e^{\lambda^mM}$ and
conjugates $\mathcal{R}^{m+1}(Q_{\beta})$ on
$\mathcal{O}_{\mathcal{R}^{m+1}(Q_{\beta})}\cup\{0\}$
to $\mathcal{R}^m(g)$ on $\mathcal{O}_{\mathcal{R}^m(g)}\cup\{0\}$. By Lemma \ref{L3.2}, we can obtain
\begin{equation}
\label{F4.4}\phi(\Delta_{\mathcal{R}^{m+1}(Q_{\beta})})=\Delta_{\mathcal{R}^m(g)}.
\end{equation}
Without loss of generality, we suppose that $\{\phi_n\}_{n=1}^{\infty}$ converges uniformly to $\phi$.
By Lemma \ref{T1.1s} and Lemma \ref{L4.2}, we have that for any $0<r<1$ and
any nonempty open set $U\subset\Delta_{\mathcal{R}^{m+1}(Q_{\beta})}(r)$,
\begin{equation}
\label{F4.2}\liminf\limits_{n\to+\infty}{\rm dens_{U}(\Delta'_{\mathcal{R}^{m+1}(Q_{\beta_n})}(r))}\geq\frac{1}{2},
\end{equation}
where $\Delta'_{\mathcal{R}^{m+1}(Q_{\beta_n})}(r)$ is the Siegel disk of the restriction of $\mathcal{R}^{m+1}(Q_{\beta_n})$ to $\Delta_{\mathcal{R}^{m+1}(Q_{\beta})}(r)$ centering at the origin.
Then the accumulated set of the sequence of Siegel disks of $\mathcal{R}^{m+1}(Q_{\beta_n})$($n=1,2,\cdots$)
centering at 0 contains the Siegel disk of $\mathcal{R}^{m+1}(Q_{\beta})$ centering at 0.
Thus combining (\ref{F4.4}) and the fact that $\phi_n$ conjugates $\mathcal{R}^{m+1}(Q_{\beta_n})$
on the Siegel disk of $\mathcal{R}^{m+1}(Q_{\beta_n})$ centering at 0 to $\mathcal{R}^m(g_n)$ on the Siegel disk of $\mathcal{R}^m(g_n)$
centering at 0, we can obtain that the limit $\phi$ conjugates
the Siegel disk of $\mathcal{R}^{m+1}(Q_{\beta})$ centering at 0 to the Siegel disk of
$\mathcal{R}^{m+1}(g)$ centering at 0, that is,
\[\phi\comp\mathcal{R}^{m+1}(Q_{\beta})(z)=\mathcal{R}^m(g)\comp\phi(z)\]
for all $z$ belonging to the Siegel disk of $\mathcal{R}^{m+1}(Q_{\beta})$ centering at the origin.

We claim that for all $0<r<1$ and any nonempty open set $U\subset\Delta_{\mathcal{R}^m(g)}(r)$,
$$\liminf\limits_{n\to+\infty}
{\rm dens_{U}(\phi_n(\Delta'_{\mathcal{R}^{m+1}(Q_{\beta_n})}(r_{\phi^{-1}})))}\geq\frac{c(K_m)}{2^{K_m}},$$
where the real number $0<r_{\phi^{-1}}<1$ satisfies
\[\phi(\Delta_{\mathcal{R}^{m+1}(Q_{\beta})}(r_{\phi^{-1}}))=\Delta_{\mathcal{R}^m(g)}(r),\]
$K_m=e^{2\lambda^mM}$ and the number $c(K_m)$ only depends on $K_m$ with $c(K_m)\to1$ as $K_m\to1$.
Indeed, since $\phi_n$ converges uniformly to $\phi$, the composition $\phi^{-1}\comp\phi_n$ converges uniformly to the identity.
We first prove that for all $0<r<1$ and any nonempty open set $U\subset\Delta_{\mathcal{R}^{m+1}(Q_{\beta})}(r)$,
$$\liminf\limits_{n\to+\infty}
{\rm dens}_U(\phi^{-1}\comp\phi_n(\Delta'_{\mathcal{R}^{m+1}(Q_{\beta_n})}(r)))\geq\frac{c(K_m)}{2^{K_m}}.$$
We choose a sequence of pairwise disjoint open balls $\{B_j\}_{j=1}^{+\infty}$ such that
\[\cup_{j=1}^{+\infty}B_j\subset U\ {\rm and}\ {\rm area}(U\setminus\cup_{j=1}^{+\infty}B_j)=0.\]
We will prove that for any $j\geq1$,
\begin{equation}
\label{F4.1}\liminf\limits_{n\to+\infty}
{\rm dens}_{B_j}(\phi^{-1}\comp\phi_n(\Delta'_{\mathcal{R}^{m+1}(Q_{\beta_n})}(r)))\geq\frac{c(K_m)}{2^{K_m}}.
\end{equation}
Let $A_j$ be an affine such that $A_j(B_j)=\mathbb{D}$. Then (\ref{F4.1}) is equivalent to
\begin{equation*}
\liminf\limits_{n\to+\infty}
{\rm dens}_{D}(A_j\comp\phi^{-1}\comp\phi_n\comp A_j^{-1}(A_j(\Delta'_{\mathcal{R}^{m+1}(Q_{\beta_n})}(r))))
\geq\frac{c(K_m)}{2^{K_m}}.
\end{equation*}
By (\ref{F4.2}), it is easy to see that
\begin{equation}
\label{F4.3}\liminf\limits_{n\to+\infty}
{\rm dens}_{D}(A_j(\Delta'_{\mathcal{R}^{m+1}(Q_{\beta_n})}(r)))\geq\frac{1}{2}.
\end{equation}
By Riemann Mapping Theorem, there exists a Riemann mapping
\[\psi_j:A_j\comp\phi^{-1}\comp\phi_n\comp A_j^{-1}(\mathbb{D})\to\mathbb{D}\]
with $\psi_j(A_j\comp\phi^{-1}\comp\phi_n\comp A_j^{-1}(0))=0$ and the derivative
$\psi_j'(A_j\comp\phi^{-1}\comp\phi_n\comp A_j^{-1}(0))>0$. Since $A_j\comp\phi^{-1}\comp\phi_n\comp A_j^{-1}$ converges uniformly to the identity on Riemann sphere with respect to the spherical metric, the domain
$A_j\comp\phi^{-1}\comp\phi_n\comp A_j^{-1}(\mathbb{D})$ converges to $\mathbb{D}$ in the sense of
Carath\'eodory. Thus by Carath\'eodory Theorem we have $\psi_j^{-1}$ converges uniformly to the identity on every compact subset of $\mathbb{D}$. Let
\[\Gamma_j=\psi_j\comp A_j\comp\phi^{-1}\comp\phi_n\comp A_j^{-1}.\]
Then $\Gamma_j$ maps $\mathbb{D}$ onto $\mathbb{D}$ with $\Gamma_j(0)=0$ and dilatation not exceeding $K_m$. According to [Theorem 1.1, \cite{AS}] and (\ref{F4.3}), we have that
$$\liminf\limits_{n\to+\infty}
{\rm dens}_{D}(\Gamma_j(A_j(\Delta'_{\mathcal{R}^{m+1}(Q_{\beta_n})}(r))\cap\mathbb{D}))\geq\frac{c(K_m)}{2^{K_m}}.$$
Then since $\psi_j^{-1}$ converges uniformly to the identity on every compact subset of $\mathbb{D}$,
one can further obtain
$$\liminf\limits_{n\to+\infty}
{\rm dens}_{D}(\psi_j^{-1}\comp \Gamma_j(A_j(\Delta'_{\mathcal{R}^{m+1}(Q_{\beta_n})}(r))\cap\mathbb{D}))\geq\frac{c(K_m)}{2^{K_m}},$$
that is,
$$\liminf\limits_{n\to+\infty}
{\rm dens}_{D}(A_j\comp\phi^{-1}\comp\phi_n\comp
A_j^{-1}(A_j(\Delta'_{\mathcal{R}^{m+1}(Q_{\beta_n})}(r))))\geq\frac{c(K_m)}{2^{K_m}}.$$
Thus (\ref{F4.1}) holds. It follows easily from (\ref{F4.1}) that
\[\liminf\limits_{n\to+\infty}
{\rm dens}_{\cup_{j=1}^{+\infty}B_j}(\phi^{-1}\comp\phi_n(\Delta'_{\mathcal{R}^{m+1}(Q_{\beta_n})}(r)))
\geq\frac{c(K_m)}{2^{K_m}}\]
and hence
\[\liminf\limits_{n\to+\infty}
{\rm dens}_{U}(\phi^{-1}\comp\phi_n(\Delta'_{\mathcal{R}^{m+1}(Q_{\beta_n})}(r)))
\geq\frac{c(K_m)}{2^{K_m}}.\]
According to Lemma \ref{L4.1}, we have that for all $0<r<1$ and any nonempty open set
$U\subset\phi(\Delta_{\mathcal{R}^{m+1}(Q_{\beta})}(r))$,
\begin{equation}
\label{c1}\liminf\limits_{n\to+\infty}
{\rm dens}_{U}(\phi_n(\Delta'_{\mathcal{R}^{m+1}(Q_{\beta_n})}(r)))
\geq\frac{c(K_m)}{2^{K_m}}.
\end{equation}
Taking $r=r_{\phi^{-1}}$ in (\ref{c1}), we obtain that for all $0<r<1$ and
any nonempty open set $U\subset\Delta_{\mathcal{R}^m(g)}(r)$,
\[\liminf\limits_{n\to+\infty}
{\rm dens}_{U}(\phi_n(\Delta'_{\mathcal{R}^{m+1}(Q_{\beta_n})}(r_{\phi^{-1}})))
\geq\frac{c(K_m)}{2^{K_m}}.\]
Thus we complete the proof of the claim.

For all $0<r<r'<1$ and any nonempty open set $U\subset\Delta_{\mathcal{R}^m(g)}(r)$,
since $\phi_n$ converges uniformly to $\phi$ and $\phi_n$ conjugates $\mathcal{R}^{m+1}(Q_{\beta_n})$ on $\Delta_{\mathcal{R}^{m+1}(Q_{\beta_n})}$ to $\mathcal{R}^m(g_n)$ on $\Delta_{\mathcal{R}^m(g_n)}$,
we have that for large enough $n$,
\[\phi_n(\Delta'_{\mathcal{R}^{m+1}(Q_{\beta_n})}(r_{\phi^{-1}}))\subset
\Delta'_{\mathcal{R}^m(g_n)}(r').\]
Then by the above claim, we have
$$\liminf\limits_{n\to+\infty}
{\rm dens_{U}(\Delta'_{\mathcal{R}^m(g_n)}(r'))}\geq\liminf\limits_{n\to+\infty}
{\rm dens}_{U}(\phi_n(\Delta'_{\mathcal{R}^{m+1}(Q_{\beta_n})}(r_{\phi^{-1}})))
\geq\frac{c(K_m)}{2^{K_m}}.$$
This will imply that for all $0<r<1$ and any nonempty open set
$U\subset\Delta_{\mathcal{R}^m(g)}(r)$,
$$\liminf\limits_{n\to+\infty}
{\rm dens_{U}\Delta'_{\mathcal{R}^m(g_n)}(r)}\geq\frac{c(K_m)}{2^{K_m}}.$$
Together with Lemma \ref{L4.2}, one can get that for any $0<r<1$ and any nonempty open set
$U\subset\Delta_{g}(r)$,
$$\liminf\limits_{n\to+\infty}
{\rm dens_{U}(\Delta'_{g_n}(r))}\geq\frac{c(K_m)}{2^{K_m}}.$$
Let $m\to+\infty$, the proof of the theorem is finished.
\end{proof}

\section{The proof of Theorem \ref{T239}}
Let us first recall three main steps of Buff and Ch\'eritat to construct quadratic Cremer Julia sets with positive area.

\vspace{0.1cm}
\noindent (a) Theorem B

\vspace{0.1cm}
\noindent (b) Let $\alpha_n$ ($n\geq1$) and $\alpha$ be the same as those in
Theorem B. If $\alpha_n\in\mathbb{HT}_N$ ($n\geq1$) and $\alpha\in\mathbb{HT}_N$,
%Let $\alpha_n\in\mathbb{HT}_N$ ($n\geq1$) and $\alpha\in\mathbb{HT}_N$ be Brjuno numbers
%such that $\alpha_n\to\alpha$ as $n\to+\infty$.
then for all $\delta>0$, if $n$ is large enough, then the postcritical set $\mathcal{O}_{P_{\alpha_n}}$ of $P_{\alpha_n}$ is contained in the $\delta$-neighborhood of the Siegel disk $\Delta_{P_{\alpha}}$ of $P_{\alpha}$.

\vspace{0.1cm}
\noindent (c) Assume that $\alpha$ is a bounded type irrational number and $\delta>0$.
We denote by $K(\delta)$ the set of points whose orbit under iteration of $P_{\alpha}$
remains at distance less than $\delta$ from $\Delta_{P_{\alpha}}$.
Then every point in the boundary of $\Delta_{P_{\alpha}}$ is a Lebesgue density point of $K(\delta)$
and Julia set $J(P_{\alpha})$ has zero area.

\vspace{0.1cm}
If one wants to construct a polynomial having a Cremer Julia set with positive area in $\mathcal{F}$ based on
the method of Buff and Ch\'eritat, then he has to establish the above three results for $\mathcal{F}$.
In fact, Theorem \ref{T20} is a key step in this paper which is indeed (a) for $\mathcal{F}$; it follows from
the proof of Buff and Ch\'eritat[p.698-p.724,\cite{BC}] and G. Zhang's result[Main Theorem,\cite{ZGF1}] that (b) can be established for $\mathcal{F}$; (c) is McMullen's result\upcite{McM98}, that is established for quadratic polynomials. However, based on the proof of McMullen and characteristics of $\mathcal{F}$, we can establish it for $\mathcal{F}$ (see Appendix B). For all $\theta\in\mathbb{R}/\mathbb{Z}$, we let $K_{\theta}$ denote the filled Julia set of $f_{\theta}(z)=e^{2\pi i\theta}z(z+1)^2$. Then based on these results, we apply the method of Buff and Ch\'eritat [p.724-p.732,\cite{BC}] to $\mathcal{F}$ and can obtain the following proposition.

\begin{proposition}
	\label{P001}Let $\alpha:=[a_0,a_1,a_2,\cdots]\in\mathbb{HT}_N$ and $p_n/q_n$ be the approximants to $\alpha$.
For all $n\geq1$, we set
	\[\alpha_n:=[a_0,a_1,\cdots,a_n,A_n,N,N,\cdots],\]
where $(A_n)$ is a sequence of positive integers such that
	\begin{equation*}
	\sqrt[q_n]{ A_n}\xrightarrow[n\to+\infty]{}+\infty
	\ {\rm and}\
	\sqrt[q_n]{\log A_n}\xrightarrow[n\to+\infty]{}1.
	\end{equation*}
Then for all $\epsilon>0$, if $n$ is large enough,
	\[{\rm area}(K_{\alpha_n})\geq(1-\epsilon)\cdot{\rm area}(K_{\alpha}).\]
\end{proposition}

Next, we use Proposition \ref{P001} to prove Theorem \ref{T239}.
Let $\theta_0=[0,a_1,a_2,\cdots]\in\mathbb{HT}_N$ be a Brjuno number.
We choose a sequence $\{\epsilon_n\}_{n=1}^{+\infty}$ of positive numbers such that
$0<\epsilon_n<1$ and $\prod_{n=1}^{+\infty}(1-\epsilon_n)$ converges.
For all strictly increasing sequences of positive integers
$\{m_j\}_{j=1}^{+\infty}$
and
$\{A_n\}_{n=1}^{+\infty}$,
we set
$\theta=[0,b_1,b_2,\cdots]$,
where
\[b_j=\left\{\begin{matrix}
a_j,&1\leq j\leq m_1,\\
A_t,&j=m_t+1,\ t\in\{1,2,\cdots\},\\
N,&else
\end{matrix}\right.\]
and for all $l\geq1$,
set $\theta_l=[0,b_1,b_2,\cdots,b_{m_l+1},N,N,\cdots]$.

According to Proposition \ref{P001}, we can choose sequences
$\{m_j\}_{j=1}^{+\infty}$
and
$\{A_n\}_{n=1}^{+\infty}$
so that
$${\rm area}(K_{\theta_l})\geq (1-\epsilon_l)\cdot {\rm area}(K_{\theta_{l-1}})\ (l\geq1)$$
and
$\sqrt[q_{m_j}]{ A_j}\xrightarrow[j\to+\infty]{}+\infty,$
where $q_{m_j}=[0,b_1,b_2,\cdots,b_{m_j}]$ ($j\geq1$).
Thus for all $l\geq1$,
\begin{equation}
\label{bf1}{\rm area}(K_{\theta_l})\geq
\prod_{j=1}^l(1-\epsilon_j)\cdot{\rm area}(K_{\theta_0}).
\end{equation}
Since $\theta_l\to\theta$ as $l\to+\infty$ and upper semi-continuous dependence on filled Julia sets of polynomials\upcite{BC},
we have
$${\rm area}(K_{\theta})\geq
\limsup_{l\to+\infty}{\rm area}(K_{\theta_l}).$$
Thus by (\ref{bf1}), we obtain
$${\rm area}(K_{\theta})\geq
\prod_{j=1}^{+\infty}(1-\epsilon_j)\cdot{\rm area}(K_{\theta_0})>0.$$
Next, we prove that the origin is a Cremer point of $f_{\theta}$.
Let us first see the Brjuno series of $\theta$.
Since
$$\frac{\log(q_{m_j}A_j)}{q_{m_j}}
=\frac{\log(q_{m_j})}{q_{m_j}}
+\log(\sqrt[q_{m_j}]{A_j})
\xrightarrow[j\to+\infty]{}+\infty,$$
the Brjuno series of $\theta$ diverges.
Assume that $f_{\theta}$ has a Siegel disk at the origin.
It follows from holomorphic motions that
every element in $\mathcal{IS}_{\theta}$ has a Siegel disk at the origin.
Let $\alpha=[0,N,b_1,b_2,\cdots]$.
Then by the discussion in Section \ref{S5.2}, we know that
the near-parabolic renormalization $\mathcal{R}(Q_{\alpha})$ of $Q_{\alpha}$
belongs to $\mathcal{IS}_{\theta}$. Thus $\mathcal{R}(Q_{\alpha})$ has a Siegel disk at the origin.
Then by Claim $B$ in the proof of Lemma \ref{p3.2}, we have that
the intersection of $\mathcal{C}_{Q_{\alpha}}\cup\mathcal{C}_{Q_{\alpha}}^{\#}$ and some small neighborhood
of the origin is contained in the filled Julia set $K(Q_{\alpha})$ of $Q_{\alpha}$.
This implies the interior of $K(Q_{\alpha})$ is not empty and hence $Q_{\alpha}$ has a Siegel disk at the origin.
This would contradict Yoccoz's famous result\upcite{Yo}:
for all $\beta\in\mathbb{R}/\mathbb{Z}$, $Q_{\beta}$ has a Siegel disk at the origin if and only if
$\beta$ is a Brjuno number. Thus the origin is a Cremer point of $f_{\theta}$.
To complete the proof, we only need to prove $J(f_{\theta})=K_{\theta}$.
Indeed, if not, then there exists a $q$-periodic Siegel cycle
$z_1,z_2,\cdots,z_q$ of $f_{\theta}$ with $f_{\theta}(z_j)=z_{j+1}$ for $j=1,2,\cdots,q-1$ and $f_{\theta}(z_q)=z_1$. According to Shishikura's method\upcite{S87}, one can obtain that
there exists a quasiconformal map $H$ such that
\begin{itemize}
\item $H(0)=0$, $H(\infty)=\infty$ and $H(z_j)=z_j$ for $j=1,2,\cdots,q$
\item $H$ is analytic in neighborhoods of $0$ and $z_j$ ($j=1,2,\cdots,q$) and derivatives $0<H'(0)<1$ and
$0<H'(z_j)<1$ ($j=1,2,\cdots,q$).
\item There exists an invariant conformal structure under $f_{\theta}\comp H$.
\end{itemize}
Then there exists a quasiconformal map $\phi$ such that $g=\phi\comp f_{\theta}\comp H\comp\phi^{-1}$ is
a cubic polynomial. It is easy to check that $g$ has a geometric attractive fixed point $\phi(0)$
and a geometric attractive cycle $\phi(z_1),\phi(z_2),\cdots,\phi(z_q)$. Thus $g$ has at least
two non-pre-periodic finite critical points. This contradicts the fact that $g$ has only a non-pre-periodic finite critical point $\phi\comp H^{-1}(-\frac{1}{3})$.

\section{Appendix A}
\begin{proof}[Proof of Claim A]
For convenience, the following notations are listed.
\begin{itemize}
\item $\mathcal{C}_h^*:=\{z\in\mathcal{P}_h:1/2
       \leq {\rm Re}(\Phi_h(z))\leq 3/2, {\rm Im}(\Phi_h(z))<-2\}$
\item $\mathcal{D}_h:=\mathcal{C}_h\cup\mathcal{C}_h^{\#}$
\item $\mathcal{D}_h^m:=\{z\in\mathcal{P}_h:1/2
       <{\rm Re}(\Phi_h(z))<1, {\rm Im}(\Phi_h(z))>-2\}$
\item $\mathcal{D}_h^r:=\{z\in\mathcal{P}_h:{\rm Re}(\Phi_h(z))=3/2,
       {\rm Im}(\Phi_h(z))\geq-2\}$
\item $\mathcal{D}_h^l:=\{z\in\mathcal{P}_h:{\rm Re}(\Phi_h(z))=1/2,
       {\rm Im}(\Phi_h(z))\geq-2\}$
\item $\mathcal{D}_h^L:=\{z\in\mathcal{P}_h:3/2<{\rm Re}(\Phi_h(z))<2,
       {\rm Im}(\Phi_h(z))>-2\}$
\item $\mathcal{D}_h^d:=\{z\in\mathcal{P}_h:1/2
       \leq {\rm Re}(\Phi_h(z))\leq 3/2, {\rm Im}(\Phi_h(z))=-2\}$
\item $\mathcal{D}_h^{-k}:=(\mathcal{C}_h)^{-k}\cup(\mathcal{C}_h^{\#})^{-k}$ for $1\leq k\leq{\bf k}_1$
\item $\mathcal{D}_h^k:=h^k(\mathcal{D}_h)$ for $0\leq k\leq k_h$,
      where $k_h$ is the integer closest to $1/\alpha(h)-{\bf k}-3$.
\item $\mathcal{D}:=\cup_{k=-{\bf k}_1}^{k_h}\mathcal{D}_h^k$
(Note that $\mathcal{D}=\cup_{k=-{\bf k}_1+1}^{k_h}\mathcal{D}_h^k$)
\end{itemize}
We first give the following properties.

\vspace{0.1cm}
\noindent{\bf Property 6.1}
\begin{itemize}
\item[(a)] $\mathcal{D}\cup\{0\}$ is closed and compactly contained in ${\rm Def}(h)$;
\item[(b)] $\mathcal{O}_h\subset\mathcal{D}$;
\item[(c)] $\Delta_h$ does't intersect $\mathcal{D}_h^d$;
\item[(d)] For all $z, z'\in\mathcal{D}_h$, if $z'=h(z)$, then $z\in\Delta_h$ if and only if $z'\in\Delta_h$.
\end{itemize}
\begin{proof}
(a): According to Inou-Shishikura, it is obvious that (a) holds.

\vspace{0.1cm}
\noindent(b): Since the critical point of $\mathcal{R}(h)$ can be iterated infinitely many times,
it is easy to see that the critical orbit of $h$ is contained in $\mathcal{D}$.
Thus together with (a), we obtain (b).

\vspace{0.1cm}
\noindent(c): First, we prove
\begin{equation}
\label{F6.1}\mathcal{C}_h^*\cap\mathcal{D}_h^{-k}=\emptyset
\end{equation}
for $1\leq k\leq{\bf k}_1-1$.
Indeed, if there exists $1\leq k_0\leq{\bf k}_1-1$ such that
$\mathcal{C}_h^*\cap\mathcal{D}_h^{-k_0}\not=\emptyset$, then
$$h^{k_0}(\mathcal{C}_h^*)\cap\mathcal{D}_h=
h^{k_0}(\mathcal{C}_h^*)\cap h^{k_0}(\mathcal{D}_h^{-k_0})\not=\emptyset$$
and this is impossible. By (\ref{F6.1}), we see easily that $\mathcal{C}_h^*\cap\mathcal{D}=\emptyset$.
Thus the property (b) tells us that $\mathcal{C}_h^*$ does't intersect $\mathcal{O}_h$.
By Lemma \ref{L3.2}, we have $\mathcal{C}_h^*\cap\partial\Delta_h=\emptyset$.
Next we prove (c). If (c) does't hold, then $\Delta_h$ intersects $\mathcal{D}_h^d$
and hence intersects $\mathcal{C}_h^*$. It follows from $\mathcal{C}_h^*\cap\partial\Delta_h=\emptyset$ that
$\mathcal{C}_h^*\subset\Delta_h$. This implies the nonzero fixed point $\sigma(h)\in\overline{\Delta}_h$.
It is easy to see that $\sigma(h)\not\in\mathcal{D}$.
Thus $\sigma(h)\not\in\mathcal{O}_h$. Then
we have $\sigma(h)\in\Delta_h$ and this is impossible.

\vspace{0.1cm}
\noindent(d): The necessity is obvious. Let us prove the sufficiency.
Since $h\in\mathcal{IS}_{\alpha}$ or $h=Q_{\alpha}$,
$h^{-1}(\mathcal{D}_h^L)$ has a component $\mathcal{D}_h^m$ contained in $\mathcal{D}_h$ and
other components with their boundaries intersecting
$\partial(V_h)$ or $-\frac{16}{27}e^{-2\pi i\alpha}$.

We will prove that any component $L$ of $h^{-1}(\mathcal{D}_h^L)$, not $\mathcal{D}_h^m$, doesn't
intersect $\mathcal{D}$. In fact, suppose that $L\cap\mathcal{D}\not=\emptyset$, that is,
$L\cap\mathcal{D}_h^k\not=\emptyset$ for some $-{\bf k}_1+1\leq k\leq k_h$. Since
$$h(L)\cap h(\mathcal{D}_h\setminus{\mathcal{D}_h^m})=
\mathcal{D}_h^L\cap(\mathcal{D}_h^1\setminus{\mathcal{D}_h^L})=\emptyset,$$
$$h(L)\cap h(\mathcal{D}_h^k)=\mathcal{D}_h^L\cap\mathcal{D}_h^{k+1}=\emptyset$$
for $1\leq k\leq k_h$ and
$$h^{-k}(L)\cap h^{-k}(\mathcal{D}_h^k)=h^{-k}(L)\cap\mathcal{D}_h=\emptyset$$
for $-{\bf k}_1+1\leq k\leq -1$, we have that $L\cap(\mathcal{D}\setminus\mathcal{D}_h^m)=\emptyset$.
Thus $L\cap\mathcal{D}_h^m\not=\emptyset$ and this is impossible.

Assume $z, z'\in\mathcal{D}_h$ with $z'=h(z)$ and $z'\in\Delta_h$. Then to prove the sufficiency,
we only need to prove $z\in\Delta_h$.
If $z\not\in\Delta_h$, we first observe that $z'\in\mathcal{D}_h^r$ and $z\in\mathcal{D}_h^l$,
then there exists a component $L$ of $h^{-1}(\mathcal{D}_h^L)$, that isn't $\mathcal{D}_h^m$,
such that $L\cap\Delta_h\not=\emptyset$.
According to (b) and Lemma \ref{L3.2}, we have $\partial\Delta_h\subset\mathcal{D}$.
Together with the above result, one can get that
$L$ is contained in $\Delta_h$ and hence the closure of $\Delta_h$ intersects
$\partial(V_h)$ or $-\frac{16}{27}e^{-2\pi i\alpha}$.
If $h\in\mathcal{IS}_{\alpha}$, this contradicts Lemma \ref{L3.2}, while if $h=Q_{\alpha}$,
this contradicts the fact that $-\frac{16}{27}e^{-2\pi i\alpha}$ is a preimage of 0 under $h$.
\end{proof}

It follows from (c) and (d) that $\pi_h(\mathcal{D}_h\cap\Delta_h)$ is an open set,
not intersecting $\pi_h(\mathcal{D}_h^d)$,
and
\begin{equation*}
\label{f3}\partial(\pi_h(\mathcal{D}_h\cap\Delta_h))=\pi_h(\mathcal{D}_h\cap\partial\Delta_h).
\end{equation*}
It is obvious that $\pi_h(\mathcal{D}_h\cap\Delta_h)$ contains a neighborhood of the `0' end.
Next we only need to prove that $\pi_h(\mathcal{D}_h\cap\Delta_h)$ is connected. If not, we let $E$ be
a component of $\pi_h(\mathcal{D}_h\cap\Delta_h)$ not containing the neighborhood of the `0' end.
Let $\phi_h:\mathbb{D}\to\Delta_h$ be a conformal map with $\phi_h(0)=0$ and
$h(\phi_h(z))=\phi_h(e^{2\pi i\alpha}z)$. We define
$$\rho:E\to\mathbb{R}^+, z\mapsto|\phi_h^{-1}(\pi_h^{-1}(z))|.$$
It is easy to check that the map $\rho$ is well defined, continuous and open and
that $\rho(z)\to1$ as $z\to\partial E$. Then $\rho$ is constant and this is impossible.
\end{proof}

\section{Appendix B\label{B}}
In this appendix, we will briefly explain that some of McMullen's treatments of quadratic polynomials in \cite{McM98} are also applicable to the family $\mathcal{F}$ of cubic polynomials, and then we can get the following two propositions.

\begin{proposition}
\label{B1}Let $\alpha$ be a bounded type irrational number. Then Hausdorff dimension of Julia set of the cubic polynomial $f_{\alpha}$ is strictly less than $2$. As a consequence, the area of Julia set of
$f_{\alpha}$ is zero.
\end{proposition}

\noindent For all $\alpha\in\mathbb{R}/\mathbb{Z}$,
the polynomial $f_{\alpha}$ has two finite critical points $-1$
and $c_0=-\frac{1}{3}$ with the corresponding critical values $0$ and $c_1=-\frac{4}{27}e^{2\pi i\alpha}$.
If $\alpha$ is a bounded type irrational number, then according to \cite{ZGF1},
we have that the Siegel disk $\Delta_{\alpha}$ of $f_{\alpha}$ at $0$ is a quasi-disk, and
the critical point $c_0$ belongs to the boundary $\partial\Delta_{\alpha}$.
Preimage of Siegel disk $\Delta_{\alpha}$ under $f_{\alpha}$ has two connected components, that is,
$\Delta_{\alpha}$ and the component $D'_{\alpha}$ containing $-1$. It is easy to see
$\overline{D'}_{\alpha}\cap\overline{\Delta}_{\alpha}=\{c_0\}$.
For all $\delta>0$, we set
\[K(\delta)=\{z\in\mathbb{C}|{\rm for\ all}\ n\geq0,\
f_{\alpha}^n(z)\ {\rm belongs\ to\ the}\ \delta-{\rm neighborhood\ of}\ \Delta_{\alpha}\}.\]
Then we have the following proposition.

\begin{proposition}
\label{B2}Assume that $\alpha$ is a bounded type irrational number and $\delta>0$. Then
every point of the boundary of $\Delta_{\alpha}$ is a Lebesgue density point of $K(\delta)$.
\end{proposition}

For convenience, we introduce some notations in \cite{McM98}.

\begin{itemize}
\item $A=O(B)$ denotes $A<CB$ for some implicit positive constant $C$.
\item $A\asymp B$ denotes $B/C<A<CB$ for some implicit positive constant $C$.
\item $d(x,y)$ denotes the Euclid distance of $x$ and $y$.
%±íʾƽÃæÉϵÄŷʽ¶ÈÁ¿£¬
%$d_U(x,y)$±íʾÇøÓò$U\subset\hat{\mathbb{C}}$ÉÏ
%µÄË«Çú¶ÈÁ¿.
%Éè$U$ÊÇ$\mathbb{C}$ÉϵÄÒ»¸öÇøÓò£¬
%¶ÔÓÚÈÎÒâµÄ$u\in U$£¬ÎÒÃdzÆ$(U,u)$ÊÇÒ»¸öµãÅÌ.
\item For any point disk $(U,u)\subset\mathbb{C}$, we set
\[{\rm in-radius}(U,u)=\sup\{r:B(u,r)\subset U\}.\]
%\[{\rm out-radius}(U,u)=\inf\{r:B(u,r)\supset U\}.\]
\end{itemize}
Next, we will give a key lemma.
The quadratic polynomial version of this lemma is the main result of Section $3$ in \cite{McM98}.
As McMullen says, the lemma tells us that the geometry of Julia set near the critical point is replicated with bounded distortion everywhere in the Julia set.
We let $\alpha$ be a bounded type irrational number. Then we have

\begin{lemma}
\label{B3}For all $z\in J(f_{\alpha})$ and $r>0$, there exists a univalent map
\[f_{\alpha}^i:(U,y)\to(V,c_0),\ i\geq0,\] such that
${\rm in-radius}(U,y)\asymp r$ and $|y-z|=O(r)$.
\end{lemma}

\begin{proof}
The treatment of McMullen on quadratic polynomials is still applicable to our case.
Thus instead of giving detailed proof here, we will just say what needs attention and modification in McMullen's method
because of changing quadratic polynomials into cubic polynomials.
We first observe that the quadratic polynomial $P_{\alpha}=e^{2\pi i\alpha}z+z^2$ has a nice symmetry, that is,
$P_{\alpha}(z)=P_{\alpha}(2c_3-z)$, where $c_3$ is the unique finite critical point of $P_{\alpha}$.
We write $\iota(z)=2c_3-z$. This symmetry brings convenience to some discussions,
for example, the Siegel disk and anthor component of its preimage are symmetric about the critical point $c_3$;
$P_{\alpha}$ is univalent in any Euclid disk not containing $c_3$.
The cubic polynomial $f_{\alpha}$ has no this symmetry, but because of the similarity near Siegel disks of $f_{\alpha}$ and $P_{\alpha}$,
we can get a similar symmetry with bounded distortion.
The first thing we should pay attention to is that $\iota$ is used in the first half of the proof of
[Proposition 3.5,\cite{McM98}]. We will replace $\iota$ with the following map $\iota_1$.
Let $\iota_1$ be a holomorphic map defined in some neighborhood $U_1$ of
$\overline{D'}_{\alpha}\cup\overline{\Delta}_{\alpha}$ such that
$\iota_1(c_0)=c_0$, $\iota_1(z)\not=z$ ($z\in U_1\setminus\{c_0\}$)
and
$f_{\alpha}(z)=f_{\alpha}(\iota_1(z))$ ($z\in U_1$).
It is easy to prove that the map $\iota_1$ exists, and the following properties hold.

\begin{itemize}
\item There exists $r>0$ such that for all $z\in\partial D'_{\alpha}\cup\partial\Delta_{\alpha}$,
we have $B(z,r)\subset U_1$.
\item For all $0<s<r$ and $z\in\partial D'_{\alpha}\cup\partial\Delta_{\alpha}$,
we have
$${\rm in-radius}(\iota_1(B(z,s)),\iota_1(z))
\asymp s.$$
\item For all $z\in\partial D'_{\alpha}$, we have
$$d(z,\partial\Delta_{\alpha})\asymp
d(\iota_1(z),\partial D'_{\alpha});$$
for all $z\in\partial\Delta_{\alpha}$, we have
$$d(z,\partial D'_{\alpha})\asymp
d(\iota_1(z),\partial\Delta_{\alpha}).$$
\end{itemize}
By the definition and properties of $\iota_1$, it is easy to see that $\iota$ can be replaced by $\iota_1$.

The second thing we should pay attention to is `$f|_{B_{j-1}}$ is univalent' in the proof of [Theorem 3.2,\cite{McM98}].
We observe that `$f|_{B_{j-1}}$ is univalent' needs the symmetry of $f=P_{\alpha}$.
In fact, according to the proof of [Theorem 3.2,\cite{McM98}],
we see easily that the sentences
\[`\epsilon>l(v_{j-1})\asymp\frac{|v_{j-1}|}{d(z_{j-1},P(f))}\]
implies that there exists $B_{j-1}=B(z_{j-1},s)$ such that $B_{j-1}\cap P(f)=\emptyset$ and
$s\asymp|v_{j-1}/\epsilon|$. Then $f|_{B_{j-1}}$ is univalent.'
can be replaced by the sentence
\[`\epsilon>l(v_{j-1})\asymp\frac{|v_{j-1}|}
{d(z_{j-1},P(f_{\alpha}))}\]
implies that there exists $B_{j-1}=B(z_{j-1},s)$ such that
$B_{j-1}\cap P(f_{\alpha})=\emptyset$, $s\asymp|v_{j-1}/\epsilon|$ and
$f_{\alpha}|_{B_{j-1}}$ is univalent.'.
And the latter follows the following claim.

\vspace{0.1cm}
\noindent\textbf{Claim:} For all $z\in J(f_{\alpha})\setminus\partial\Delta_{\alpha}$,
there exists a ball $B(z,r)$ such that $B(z,r)\cap\partial\Delta_{\alpha}=\emptyset$,
$d(z,\partial\Delta_{\alpha})\asymp r$ and $f_{\alpha}$ is univalent in $B(z,r)$.

Indeed, in a sufficiently small neighborhood $B(c_0,s)$ of $c_0$, we have
$$f_{\alpha}(z)=(\phi(z+c_0))^2+c_1\ {\rm for}\ z\in B(c_0,s),$$
where $\phi$ is a univalent map with $\phi(0)=0$.
Thus one obtains easily that for all $z\in B(c_0,s/2)$, there exists a ball $B(z,s_1)\subset B(c_0,s)$
such that $c_0\not\in B(z,s_1)$, $s_1\asymp |z-c_0|$ and $f_{\alpha}$ is univalent in $B(z,s_1)$.
Thus the claim holds for $z\in(J(f_{\alpha})\setminus\partial\Delta_{\alpha})\cap B(c_0,s/2)$.
Since the compact set $J(f_{\alpha})\setminus B(c_0,s/2)$ does't contain any critical point of $f_{\alpha}$,
we have that there exists $r_1>0$ such that for all $z\in J(f_{\alpha})\setminus B(c_0,s/2)$,
$f_{\alpha}$ is univalent in $B(z,r_1)$. It follows from the boundedness of $J(f_{\alpha})\setminus B(c_0,s/2)$
that the claim holds for $z\in (J(f_{\alpha})\setminus\partial\Delta_{\alpha})-B(c_0,s/2)$.
This completes the proof of the claim.
\end{proof}

Based on Lemma \ref{B3}, we can get that $J(f_{\alpha})$ is shallow in the same way as that in
the proof of [Theorem 4.1,\cite{McM98}]. Let us recall a compact set $\Lambda\subset\mathbb{C}$
is shallow if and only if for all $z\in\Lambda$ and $0<r<1$, there exists a ball $B\subset\mathbb{C}$
such that $B\cap\Lambda=\emptyset$, ${\rm diam} B\asymp r$, $d(z,B)=O(r)$. If the compact set
$\Lambda$ is shallow, then one can see that Hausdorff dimension ${\rm H\cdot dim}(\Lambda)$
of $\Lambda$ is strictly less than $2$. Thus we obtain Proposition \ref{B1}.

Next, let us recall deep points and measurable deep points.
Let $\Lambda\subset\mathbb{C}$ be a compact set, for all $z\in\Lambda$,
$z$ is a deep point of $\Lambda$ if and only if there exists $\delta>0$ such that
for all $0<r<1$, we have
\[B(y,s)\subset B(z,r)-\Lambda\Rightarrow
s=O(r^{1+\delta});\]
we say that $z$ is a measurable deep point of $\Lambda$, if there exists $\delta>0$ such that
for all $r>0$, we have
\[{\rm area}(B(z,r)-\Lambda)=O(r^{2+\delta}).\]
About deep points and measurable deep points, we have the following important connection, see
[Proposition 4.3,\cite{McM98}].

\begin{lemma}
\label{B5}If $z$ is a deep point of compact set $\Lambda\subset\mathbb{C}$ and
$\partial{\Lambda}$ is shallow, then $z$ is a measurable deep point of $\Lambda$.
\end{lemma}

According to the similarity near Siegel disks of $f_{\alpha}$ and
$P_{\alpha}$, we can obtain the following result by the same way as
[Theorem 4.2,\cite{McM98}].

\begin{lemma}
\label{B6}For all $\delta>0$,
every point in $\partial\Delta_{\alpha}$ is a deep point of $K(\delta)$.
\end{lemma}

At last, by Lemma \ref{B5}, Lemma \ref{B6} and the fact that $J(f_{\alpha})$ is shallow,
we can obtain the proof of Proposition \ref{B2} by the similar way as that in the proof of
[Corollary 4.5,\cite{McM98}]. But
since $f_{\alpha}$ has critical points in $\mathbb{C}$ one more than quadratic polynomials
in $\mathbb{C}$, to prove that `$f_{\alpha}^n|_U$ is univalent', we have to analyze one more case than quadratic polynomials, that is, the case that $U$ is contained in a component of preimage of $D_{\alpha}'$ under some iterate of $f_{\alpha}$. Indeed, this follows easily from the Monodrony Theorem.

\section{Acknowledgements}
The research work was supported by National Key R\&D Program of China under Grant 2019YFB1406500.


\begin{thebibliography}{15}
\bibitem{AS} K. Astala. Area distortion of quasiconformal mappings. Acta Mathematica 173.1 (1994): 37-60.
\bibitem{AL} A. Avila and M. Lyubich. Lebesgue measure of Feigenbaum Julia sets. arXiv preprint arXiv:1504.02986 (2015).
\bibitem{B97} X. Buff. Ensembles de Julia de mesure positive (d'apr\'es van Strien et Nowicki).
S\'eminaire Bourbaki, Ast\'erisque v. 245 (1997): 7-39.
\bibitem{BC} X. Buff and A. Ch\'eritat. Quadratic Julia sets with positive area. Annals of Mathematics (2012): 673-746.
\bibitem{BR} L. Bers and H. Royden. Holomorphic families of injections. Acta Mathematica 157 (1986): 259-286.
%\bibitem{C} Cheraghi, Davoud. ``Typical orbits of quadratic polynomials with a neutral fixed point: non-Brjuno type.'' Ann. Sci. \'Ecole Norm. Sup. http://arxiv.org/abs/1001.4030 (2010, accepted)
\bibitem{C17} D. Cheraghi, Topology of irrationally indifferent attractors, arXiv: math.DS/1706. 02678 v1, 2017.
\bibitem{C1} A. Ch\'eritat. The hunt for Julia sets with positive measure, in Complex Dynamics, AK Peters, Wellesley, MA, (2009), pp. 539¨C559. MR 2508268. http://dx.doi.org/10.1201/b10617-19.
\bibitem{DH84} A. Douady and J. Hubbard. \'Etude dynamique des polyn\^omes complexes \uppercase\expandafter{\romannumeral1}. Publ. Math. d'Orsay {\bf 84-02} (1984).
\bibitem{DH85} A. Douady and J. Hubbard. \'Etude dynamique des polyn\^omes complexes \uppercase\expandafter{\romannumeral2}. Publ. Math. d'Orsay {\bf 84-04} (1985).
\bibitem{GA} F. Gardiner. Teichmiiller Theory and Quadratic Differentials, Wiley-Interscience, New York.(1987)
\bibitem{DH} A. Douady and J. Hubbard. On the dynamics of polynomial-like mappings. Annales scientifiques de l'\'Ecole normale sup\'erieure. Vol. 18. No. 2. (1985)
\bibitem{DL18} D. Dudko and M. Lyubich. Local connectivity of the Mandelbrot set at some satellite parameters of bounded type. arXiv preprint arXiv:1808.10425 (2018).
\bibitem{F1919} P. Fatou. Sur les \'equations fonctionelles, 2-me memoire. Bull. S.M.F., t. 47 (1919): p. 43.
\bibitem{FY} Y. Fu and F. Yang. Area and Hausdorff dimension of Sierpi¨½ski carpet Julia sets. Mathematische Zeitschrift (2019): 1-16.
\bibitem{HA} J. Hubbard. Teichm¨¹ller theory and applications to geometry, topology, and dynamics.(2016)
\bibitem{IS} H. Inou and M. Shishikura. The renormalization for parabolic fixed points and their perturbation. preprint (2006).
\bibitem{McM} C. McMullen. Complex Dynamics and Renormalization, Annals of Mathematics Studies. Princeton Univ. Press, Princeton, (1994).
\bibitem{McM98} C. McMullen. Self-similarity of Siegel disks and Hausdorff dimension of Julia sets. Acta Math. 180 (1998), 247¨C292. MR 1638776. Zbl 0930.37022. http://dx.doi.org/10.1007/BF02392901.
\bibitem{Lyu83} M. Y. Lyubich, Typical behavior of trajectories of the rational mapping of a sphere, Dokl. Akad. Nauk SSSR 268 (1983): 29¨C32. MR 0687919. Zbl 0595.30034.
\bibitem{Lyu84} M. Lyubich. Investigation of the stability of the dynamics of rational functions Teor. Funktsii Funktsional. Anal. i Prilozhen 42 (1984): 72-91.
\bibitem{Lyu} M. Lyubich. On the Lebesgue measure of the Julia set of a quadratic polynomial. arXiv preprint math/9201285 (1991).
\bibitem{PZ} C. L. Petersen and S. Zakeri. On the Julia set of a typical quadratic
polynomial with a Siegel disk, Ann. of Math. 159 (2004): 1¨C52. MR 2051390.
Zbl 1069.37038. http://dx.doi.org/10.4007/annals.2004.159.1.
\bibitem{R} S. Rickman. Removability theorems for quasiconformal mappings. Ann. Acad. Sci. Fenn. Ser.
    A I. 449(1969),1-8.
\bibitem{S87} M. Shishikura. On the quasiconformal surgery of rational functions. In: Annales scientifiques de l'\'Ecole Normale Sup\'erieure. (1987). p. 1-29.
\bibitem{Shi1} M. Shishikura. Topological, geometric and complex analytic properties of
Julia sets, in Proceedings of the International Congress of Mathematicians,
Vol. 1, 2 (Z\"urich, 1994), Birkh\"auser, Basel, (1995): pp. 886¨C895. MR 1403988.
Zbl 0843.30026.
\bibitem{ST00} M. Shishikura and L. Tan. An alternative proof of Ma\~n\'e's theorem on non-expanding Julia sets. London Mathematical Society Lecture Note Series (2000): 265-280.
\bibitem{SY} M. Shishikura and F. Yang. The high type quadratic Siegel disks are Jordan domains. arXiv preprint arXiv:1608.04106 (2016).
\bibitem{Yo} J. Yoccoz. Petits diviseurs en dimension 1. Ast\'erisque {\bf 231}(1995)
\bibitem{ZGF1} G. Zhang. All bounded type Siegel disks of rational maps are quasi-disks. Invent. Math.
185 (2011), no. 2, 421-466.
\end{thebibliography}
\end{document}